\documentclass[a4paper,11pt,reqno]{amsart}
\usepackage{amsfonts, amssymb, amsmath, amsthm}
\usepackage{tikz}
\usepackage{pgfplots}
\usepackage[all]{xy}
\usepackage{etex}

\usepackage{tkz-berge}

\usepackage{listings}
\usepackage{ctable}
\usepackage{mathtools}
\usepackage{microtype}
\newcommand{\otoprule}{\midrule[\heavyrulewidth]} 
\usepackage{hyperref}
\usepackage[hyphenbreaks]{breakurl}

\newtheorem{lemma}{Lemma}
\newtheorem{theorem}[lemma]{Theorem}
\newtheorem*{theorem*}{Theorem}

\newtheorem{proposition}[lemma]{Proposition}
\newtheorem{corollary}[lemma]{Corollary}

\theoremstyle{definition}
\newtheorem{definition}[lemma]{Definition}
\newtheorem*{definition*}{Definition}

\newtheorem{example}[lemma]{Example}

\newtheorem{remark}[lemma]{Remark}

\newcommand{\Z}{\mathbb{Z}}

\newcommand{\cp}{\mathbin{\Box}}
\newcommand{\join}{\star}
\newcommand{\pt}{\mathit{pt}}

\renewcommand{\geq}{\geqslant}

\DeclareMathOperator{\MH}{MH}
\DeclareMathOperator{\MC}{MC}
\DeclareMathOperator{\MS}{MS}
\DeclareMathOperator{\HH}{H}
\DeclareMathOperator{\rank}{rank}

\DeclareMathOperator{\card}{card}

\title{Categorifying the magnitude of a graph}

\author{Richard Hepworth}
\address[R.~Hepworth]{Institute of Mathematics\\
University of Aberdeen\\
Aberdeen AB24 3UE\\
United Kingdom}
\email{r.hepworth@abdn.ac.uk}

\author{Simon Willerton}
\address[S.~Willerton]{School of Mathematics and Statistics,
Hicks Building, 
University of Sheffield, 
Sheffield,
S3 7RH, UK
}
\email{s.willerton@sheffield.ac.uk}


\begin{document}

\begin{abstract}
The magnitude of a graph can be thought of as an integer power 
series associated to a graph; Leinster introduced it using his idea of 
magnitude of a metric space.  Here we introduce a bigraded 
homology theory for graphs which has the magnitude as its graded 
Euler characteristic.  This is a categorification of the magnitude in the 
same spirit as Khovanov homology is a categorification of the Jones 
polynomial.
We show how properties of magnitude proved by Leinster categorify 
to properties such as a K\"unneth Theorem and a Mayer-Vietoris 
Theorem.
We prove that joins of graphs have their homology supported on the diagonal.
Finally, we give various computer calculated examples.
\end{abstract}

\maketitle
\tableofcontents

\section{Introduction}
\subsection{Overview}
The magnitude of a finite metric space was introduced by 
Leinster~\cite{LeinsterMetricSpace} by analogy with 
his notion of the Euler characteristic of a 
category~\cite{LeinsterEulerCharCategory}.
This was found to have connections with topics as varied as
intrinsic volumes~\cite{LeinsterWillertonAsymptotic}, 
biodiversity~\cite{LeinsterCobboldDiversity}, 
potential theory~\cite{MeckesMagnitudeDimensions}, 
Minkowski dimension~\cite{MeckesMagnitudeDimensions} 
and curvature~\cite{WillertonHomogeneous}.

This invariant of finite metric spaces can be used to construct an
invariant of finite graphs.
For $G$ a finite graph and $t>0$, 
we equip the set of vertices of $G$ with 
the shortest path metric on $G$ where each edge is given length $t$.
Leinster~\cite{LeinsterGraph} showed that as a function of $t$, 
the magnitude of this metric space is a rational function in $e^{-t}$.
Writing $q=e^{-t}$, the magnitude can be expanded as a formal power series 
in $q$ and Leinster proved that this power series has integer coefficients.
It is this integer power series that we will take as the magnitude of $G$, and we will write it as $\#G$.
For example, the five-cycle graph has magnitude which starts as follows:
  \[\#C_5=5-10q+10q^2-20q^4+40q^5-40q^6-80q^8+\cdots.\] 

In this paper we will categorify the magnitude of graphs
by defining \emph{magnitude homology of graphs}. 
This is a bigraded homology theory $\MH_{\ast,\ast}$.
It is functorial with respect to maps of graphs
that send vertices to vertices and preserve or contract edges,
and its graded Euler characteristic recovers the magnitude: 
\begin{equation}\label{Categorify}
	\#G = \sum_{k,l\geqslant 0} (-1)^k\cdot \rank
	\bigl(\MH_{k,l}(G)\bigr)\cdot  q^l= \sum_{l\geqslant 0}  \chi\bigl(\MH_{\ast,l}(G)\bigr)\cdot  q^l.
\end{equation}
Thus our categorification is in exactly the same spirit as
as Khovanov's categorification
of the Jones polynomial~\cite{Khovanov} and Ozsvath-Szabo's categorification
of the Alexander polynomial~\cite{OzsvathSzabo}.
As an example, the ranks of the magnitude homology groups of the 
five-cycle graph are given in Table~\ref{TableFiveCycle}.  
You can verify that the Euler characteristic of each of 
the first few rows is the corresponding coefficient in $\# C_5$.

\begin{table}
\begin{center}
\begin{tikzpicture}[scale=0.5]
\foreach \x in {0,72,...,288}
    \draw (\x+90:2cm) -- (\x+72+90:2cm);
%\draw (0:1cm) -- (180:1cm);
\foreach \x in {0,72,...,288}
    \draw [fill=red](\x+90:2cm) circle (0.1cm);
\end{tikzpicture}
%
%\quad
%
\footnotesize
\begin{tabular}{rrrrrrrrrrrrrr}
&&&&&&&$k$\\
&&0&1&2&3&4&5&6&7&8&9&10&11\\
\otoprule                                                                 
&0 & 5\\
& 1   & &     10                   \\                                          
 &2     & &&         10             \\                                          
& 3    &&&           10  &  10       \\                                          
& 4         &&&&            30  &  10  \\                                         
& 5          &&&&&                 50   & 10                          \\           
$l$& 6          &&&&&                 20  &  70   & 10                    \\           
& 7           &&&&&&                      80  &  90  &  10              \\           
& 8           &&&&&&&                           180  & 110 &   10      \\             
& 9            &&&&&&&                          40 &  320  & 130  &  10 \\            
&10           &&&&&&&&                                 200 &  500  & 150 &   10 \\      
&11           &&&&&&&&&                                       560  & 720 &  170  &  10\\ 
\bottomrule
\smallskip
\end{tabular}
\end{center}
\caption{The ranks of $\MH_{k,l}(C_5)$, the magnitude homology groups of the pictured five-cycle graph, as computed using Sage.}
\label{TableFiveCycle}
\end{table}

Being a bigraded abelian group rather than just a power series, 
the magnitude homology has a richer structure than the magnitude.
For example, functoriality means that for a given graph its
magnitude homology is equipped with an action of its automorphism group.
We will see below that 
various properties of magnitude described by Leinster in~\cite{LeinsterGraph} 
are shadows of properties of magnitude homology.

Leinster has given a counting formula~\cite[Proposition~3.9]{LeinsterGraph}
for the magnitude.  
It expresses the coefficient of $q^{l}$ in $\#G$ as
  \[\sum_{k\geq 0}(-1)^{k}\Bigl | \bigl\{(x_{0},\dots, x_{k})\, \colon\, x_{i}\in V(G),\,x_{i}\neq x_{i+1}, \,\textstyle\sum_{i=0}^{k-1}d(x_{i},x_{i+1})=l\bigr\}\Bigr|.\]
This expression is precisely the alternating sum of the ranks 
of the magnitude chain groups,
and in general these ranks are considerably larger 
than the ranks of magnitude homology groups.  
In Table~\ref{TableFiveCycleChains}, the ranks of the magnitude chain 
groups for the five-cycle graph are given and this should be compared 
with Table~\ref{TableFiveCycle}.  
Again the Euler characteristic of each row gives a coefficient 
of the magnitude, but the terms grow exponentially as you move down diagonally.
This means that the magnitude homology groups are counting something 
much subtler than Leinster's formula is.

\begin{table}
\scriptsize
\begin{center}
\begin{tabular}{rrrrrrrrrrrrrr}
&&&&&&&$k$\\
&&0&1&2&3&4&5&6&7&8&9&10&11\\
\otoprule                                                                 
&0 & 5\\
& 1   & &     10                   \\                                          
 &2     & &10&         20             \\                                          
& 3    &&&           40  &  40       \\                                          
& 4         &&&20&            120  &  80  \\                                         
& 5          &&&& 120&                 320   & 160                          \\           
$l$& 6          &&&&40 &                 480  &  800   & 320                    \\           
& 7           &&&&& 320 &  1600 &  1920  &  640         \\           
& 8           &&&&&80  & 1600 &  4800 &  4480  & 1280              \\             
& 9            &&&&&&      800  & 6400 & 13440 & 10240 &  2560   \\            
&10           &&&&&&     160 &  4800  &22400 & 35840 & 23040 &  5120 \\      
&11           &&&&&&&      1920 & 22400 & 71680 & 92160  &51200 & 10240        \\ 
\bottomrule
\smallskip
\end{tabular}
\end{center}
\caption{The ranks of $\MC_{k,l}(C_5)$ the magnitude chain groups of the five-cycle graph, as computed using Sage.}
\label{TableFiveCycleChains}
\end{table}

\subsection{Categorifying properties of the magnitude}
Many of the properties of the magnitude that were proved
by Leinster in~\cite{LeinsterGraph} can be categorified,
meaning that they follow from properties of the magnitude
homology upon taking the graded Euler characteristic.
The categorified properties are subtler, being properties of
the homology rather than its Euler characteristic,
and are correspondingly harder to prove.
We list the categorifications here.

\subsubsection{Disjoint unions}
Leinster shows that magnitude is additive 
with respect to the disjoint union of graphs~\cite[Lemma~3.5]{LeinsterGraph}:
 \[\#(G\sqcup H) = \#G + \#H\]
Our categorification of this, Proposition~\ref{disjoint}, is the additivity of the magnitude homology:
 \[\MH_{\ast,\ast}(G\sqcup H) \cong \MH_{\ast,\ast}(G) \oplus \MH_{\ast,\ast}(H).\]
Taking the graded Euler characteristic of both sides recovers
Leinster's formula for $\#(G\sqcup H)$.

\subsubsection{Products}
Leinster shows that magnitude is multiplicative with respect
to the cartesian product $\cp$ of graphs~\cite[Lemma~3.6]{LeinsterGraph}:
 \[\#(G\cp H) = \#G\cdot \#H.\]
The categorification of this is Theorem~\ref{kunneth:theorem}, a K\"unneth Theorem which says that there is a  non-naturally split, short exact sequence:
\begin{multline*}
		 0\to
		 \MH_{\ast,\ast}(G)\otimes \MH_{\ast,\ast}(H)
		\to
		\MH_{\ast,\ast}(G\cp H)\\
		\to
		\mathrm{Tor}\bigl(\MH_{\ast+1,\ast}(G), \MH_{\ast,\ast}(H)\bigr)
		\to 0.
\end{multline*}
Taking the graded Euler characteristic of this sequence 
recovers the multiplicativity of the magnitude.  
Moreover, if either $G$ or $H$ has torsion-free magnitude homology,
then this sequence reduces to an isomorphism 
$\MH_{\ast,\ast}(G)\otimes \MH_{\ast,\ast}(H)\cong\MH_{\ast,\ast}(G\cp H)$.  
At the time of writing, despite quite a bit of computation, we don't know whether any graphs have torsion in their magnitude homology.

\subsubsection{Unions}
Magnitude can be extended to infinite metric 
spaces~\cite{MeckesPositiveDefinite} and 
the Convexity Conjecture~\cite{LeinsterWillertonAsymptotic} 
gives an explicit formula for the magnitude of compact, convex subsets of $\mathbb{R}^n$.
A corollary of the conjecture would be 
that the magnitude of compact, convex subsets satisfies an inclusion-exclusion formula.  
Leinster showed that an analogue of this corollary holds for graphs. 
 If $(X;G,H)$ is a projecting decomposition  
(see Section~\ref{SectionMayerVietoris}), 
so that in particular, $X=G\cup H$, then the inclusion-exclusion 
formula holds for magnitude~\cite[Theorem~4.9]{LeinsterGraph}:
  \[\#X = \#G +\#H - \# (G\cap H).\]
Our categorification of this result, Theorem~\ref{MV:Theorem}, is that 
if $(X;G,H)$ is a projecting decomposition,
then there is a naturally split short exact sequence
  \[0\to \MH_{\ast,\ast}(G\cap H)\to \MH_{\ast,\ast}(G)\oplus\MH_{\ast,\ast}(H)\to \MH_{\ast,\ast}(X)\to 0\] 
(which we think of as a form of Mayer-Vietoris sequence)
and consequently a natural isomorphism
   \[
   \MH_{\ast,\ast}(G)\oplus\MH_{\ast,\ast}(H)
     \cong
     \MH_{\ast,\ast}(X)\oplus  \MH_{\ast,\ast}(G\cap H).
     \]
Taking the Euler characteristic recovers the inclusion-exclusion formula for magnitude.

\subsection{Diagonality}
Leinster~\cite{LeinsternCafeTutte} noted  many  examples of graphs which had magnitude with alternating coefficients; 
these examples included complete graphs, complete bipartite graphs, forests and graphs with up to four vertices.  
This phenomenon can be explained in terms of magnitude homology.
Call a graph $G$ \emph{diagonal} if $\MH_{k,l}(G)=0$ if $k\neq l$.
In this case the formula~\eqref{Categorify} becomes
\[\#G=\sum_{l\geq 0}(-1)^{l}\cdot\rank\MH_{l,l}(G)\cdot q^{l},\]
and shows in particular that the coefficients of the magnitude 
alternate in sign.
Recall that the join $G\join H$ of graphs $G$ and $H$ is obtained by adding an edge between every vertex of $G$ and every vertex of $H$.  
This is a very drastic operation, for instance the diameter of the resulting join is at most $2$.  
We prove in Theorem~\ref{join:theorem} that any join $G\join H$ 
of non-empty graphs has diagonal magnitude homology.  
This tells us immediately that complete graphs and complete bipartite graphs are diagonal.   
Together with the other properties of magnitude homology mentioned above,
we recover the alternating magnitude property of all the graphs 
noted by Leinster, as well as many more.

  \subsection{The power series expansion and asymptotics}  It is worth
commenting here on how the magnitude of graphs fits in with the general
theory of magnitude of metric spaces.
One nice class of metric spaces, as far as magnitude is concerned, is
the class of subsets of Euclidean space  (this is a
subclass of the class of positive definite metric spaces~\cite{MeckesPositiveDefinite}).
For $A$ a finite metric space let $tA$ be $A$ with the metric scaled
by a factor of $t$.  If $X$ is a non-empty finite subset of Euclidean space, the 
\emph{magnitude function} $|tX|$ is defined for all
$t>0$, satisfies $t\ge 1$  and is continuous on $(0,\infty)$; see Corollaries~2.4.5
and~2.5.4 of~\cite{LeinsterMetricSpace} and Corollary~5.5
of~\cite{MeckesMagnitudeDimensions}.  It is not known whether the
magnitude function of such a space is increasing or not --- see the discussion
after Corollary~6.2 of~\cite{MeckesMagnitudeDimensions} --- but certainly all
computed examples are increasing.   However, it is known~\cite%[Theorems~2 and~3]
{LeinsterWillertonAsymptotic} that for any finite metric 
space $A$ and for $t\gg 0$ the magnitude $\left|tA\right|$ is defined 
and increasing in $t$, with $\left|tA\right|\to \card(A)$ as $t\to \infty$.
 A random and seemingly typical example of a subset of Euclidean space is given
 in Figure~\ref{FigurePosDefAndGraph}.

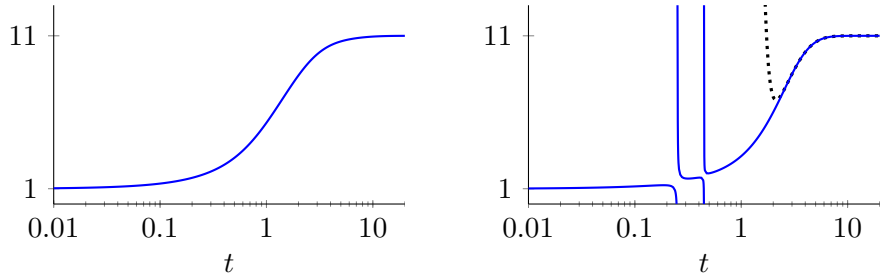
\begin{figure}[t]
\begin{center}
\begin{tikzpicture}
\begin{semilogxaxis}[%tick label style={font=\footnotesize},
width = 0.49\textwidth,
%axis equal image=true,
axis x line=bottom, axis y line = left,
xmin=0.01,ymin=0, ymax=13, xmax=20,
xtick={0.01,0.1,1,10}, xticklabels={0.01,$0.1$,$1$ ,10},
ytick={1,11}, yticklabels={1,11},
x axis line style={style = -},y axis line style={style = -},
xlabel=$t$,%ylabel=$\RMS$,
yscale=0.7,
legend style={at={(1,0.1)},anchor=south east}
]
\addplot[mark=none,blue,thick] file {PosDefMag.dat};
%\addlegendentry{$\RMS(tR)$}
%{\addplot[mark=none,black, very thin] expression {3};}
\end{semilogxaxis}
\end{tikzpicture}
\qquad
\begin{tikzpicture}
\begin{semilogxaxis}[%tick label style={font=\footnotesize},
width = 0.49\textwidth,
%axis equal image=true,
axis x line=bottom, axis y line = left,
xmin=0.01,ymin=0, ymax=13, xmax=20,
xtick={0.01,0.1,1,10}, xticklabels={0.01,$0.1$,$1$ ,10},
ytick={1,11}, yticklabels={1,11},
x axis line style={style = -},y axis line style={style = -},
xlabel=$t$,%ylabel=$\RMS$,
yscale=0.7,
]
\addplot[mark=none,black,very thick,dotted] file {RandGraphApprox200.dat};
%\addplot[mark=none,blue,thick] file {code/RandGraphMag200.dat};
\addplot[mark=none,blue,thick] file {RandGraphMag1.dat};
\addplot[mark=none,blue,thick] file {RandGraphMag2.dat};
\addplot[mark=none,blue,thick] file {RandGraphMag3.dat};
\end{semilogxaxis}
\end{tikzpicture}
\end{center}
\caption{The left hand picture shows the magnitude function for a certain $11$-point subset of $\mathbb{R}^{2}$.  
The right hand picture shows the magnitude function for a certain graph with $11$ vertices and $33$ edges.  
The dotted line shows the approximation to the magnitude function using the first seven terms of the expansion in powers of $q=e^{-t}$.}
\label{FigurePosDefAndGraph}
\end{figure}

The metric space obtained from a graph is generally not isometric to a subset of Euclidean space and  the magnitude function $\left|tG\right|$ will have many singularities, see Figure~\ref{FigurePosDefAndGraph}.  
 However, we know from the  above that the magnitude will eventually become 
 nice in that the magnitude $\left|tG\right|$ is defined and increasing to the number of vertices of $G$ as $t\to\infty$.  
 We defined the magnitude of a graph $G$ to be the formal power series about $q=0$, which, as $q=e^{-t}$, corresponds to expanding in negative exponentials near $t=\infty$, this avoids the bad behaviour of the magnitude function.  
 Again, see Figure~\ref{FigurePosDefAndGraph}. 

This perspective can be compared with previous examinations of asymptotics of the magnitude of infinite spaces in~\cite{LeinsterWillertonAsymptotic} and~\cite{WillertonHomogeneous}.  
There polynomial contributions to the aymptotics of the magnitude were shown to come from things like volume, surface area, total scalar curvature and Euler characteristic, which are obtained by integrating local phenomena such as curvature.  
Here we are looking at exponentially decaying contributions  to the asymptotics and these come from counting global phenomena, such as certain paths of a given length.

\subsection{Further directions and open questions}  
The results described above demonstrate that magnitude
homology is natural, nontrivial, and can shed light on properties
of the magnitude that are otherwise unexplained.  
Here are a few questions and avenues for further study.
\begin{itemize}
\item \label{ishomologybetter}
There are examples of non-isomorphic graphs with
isomorphic magnitude homology, for example any two trees
with the same number of vertices.
Are there graphs with the same magnitude 
but different magnitude homology groups?
\item 
Is there a graph whose magnitude homology contains torsion?
\item 
Leinster showed that if two graphs differ by a Whitney
twist with adjacent gluing points, then their magnitudes are equal.
Do two graphs related by a Whitney twist have isomorphic magnitude homology?
\item 
Prove the magnitude homology of cyclic graphs is as is conjectured in Appendix~\ref{section:CyclicGraphs}.
\item 
Computations suggest that the icosahedral graph
(i.e.~the $1$-skeleton of the icosahedron)
has diagonal homology.  
We have not been able to apply any of our techniques
for proving that graphs are diagonal in this case,
and in particular the graph is not a join.
Is the icosahedral graph diagonal, and if so why?
\item 
We anticipate that there is a theory of magnitude \emph{cohomology}
dual to the homological theory developed in this paper.
As with cohomology of spaces, it should be possible to
equip this theory with a product structure
\item
One can define $\MH_{\ast,\ast}(G)$ as the
reduced homology of a sequence of pointed simplicial
sets.  This is used in Section~\ref{KunnethProof:Section},
see in particular Remark~\ref{simplicialapproach}.
We have chosen not to emphasise this approach in the present
paper, but there may be advantages to doing so in future.
\end{itemize}

\subsection{Organisation of the paper}
The paper is organised as follows.
Sections~\ref{definition:section} and~\ref{SectionInducedMaps}
define magnitude homology as a functor and prove that
it categorifies the magnitude.  Sections~\ref{SectionDisjointUnions},
\ref{SectionCartesianProducts} and~\ref{SectionMayerVietoris}
cover the magnitude homology of disjoint unions,
cartesian products and unions, describing in detail
the categorified properties of the magnitude discussed above.
Then section~\ref{SectionDiagonalGraphs} discusses diagonal
graphs, in particular the fact that joins are diagonal.
Sections~\ref{SectionExcision}, \ref{join:proof:section}
and~\ref{KunnethProof:Section} give some lengthy
deferred proofs.  Finally, Appendix~\ref{appendix}
records and discusses some computer calculations of magnitude
homology.

\subsection{Acknowledgements} 
We would like to thank Tom Leinster for useful discussions
throughout the development of this work,
Dave Benson for some useful comments,
and James Cranch for translating our software into Sage.

\section{The definition of magnitude homology}
\label{definition:section}

In this section we define the magnitude homology of
a graph $G$, give some very basic examples and
properties, and establish the relationship between
a graph's magnitude homology and its magnitude.
First we state our conventions, which are taken
directly from~\cite{LeinsterGraph}.
By a \emph{graph} we mean a finite undirected graph
with no loops or multiple edges.  The set of vertices of a graph $G$
is denoted $V(G)$ and the set of edges is denoted $E(G)$.
If $x$ and $y$ are vertices of a graph $G$, then the \emph{distance} $d_G(x,y)$
(or simply $d(x,y)$ where it will not cause confusion)
is defined to be the length of a shortest edge path from $x$ to $y$.
If $x$ and $y$ lie in different components of $G$ then $d_G(x,y)=\infty$.
Thus $d_G$ is a metric on $V(G)$, so long as we allow metrics to
take value $\infty$ on certain pairs.

\begin{definition}
Let $G$ be a graph.
The \emph{length} of a tuple $(x_0,\ldots,x_k)$ of 
vertices of $G$ is 
\[
	\ell(x_0,\ldots,x_k) = d(x_0,x_1)+\cdots+d(x_{k-1},x_k).
\]
For $i=0,\ldots,k$ the triangle inequality
guarantees that 
\begin{equation}\label{inequality:one}
	\ell (x_0,\ldots,\widehat{x_i},\ldots,x_k)
	\leqslant
	\ell(x_0,\ldots,x_{k})
\end{equation}
where a `hat' indicates a term that has been omitted.
\end{definition}

\begin{definition}[The magnitude chain complex]
The \emph{magnitude chain complex} $\MC_{\ast,\ast}(G)$ of a graph $G$
is the direct sum of chain complexes
\[
	\bigoplus_{l\geq 0}\MC_{\ast,l}(G)
\]
where the chain complex $\MC_{\ast,l}$ is defined as follows.
The group $\MC_{k,l}(G)$
is freely generated by tuples $(x_0,\ldots,x_k)$
of vertices of $G$ satisfying 
$x_0\neq x_1\neq\cdots\neq x_k$ and
$\ell(x_0,\ldots, x_k)=l$.
The differential
\[
	\partial\colon \MC_{k,l}(G)\to \MC_{k-1,l}(G)
\]
is the alternating sum
%$\partial = \sum_{i=1}^{k-1}(-1)^i \partial_i$
$\partial=\partial_1-\partial_2+\cdots+(-1)^{k-1}\partial_{k-1}$
of the maps defined by 
\[
	\partial_{i}(x_0,\ldots,x_k)
	=
	\begin{cases}
		(x_0,\ldots,\widehat{x_i},\ldots,x_k)
		&
		\mathrm{if\ }
		\ell(x_0,\ldots,\widehat{x_i},\ldots,x_k)=l,
		\\
		0
		&
		\mathrm{otherwise}.
	\end{cases}
\]
It is shown in Lemma~\ref{differential:lemma} 
below that $\partial\circ\partial=0$,
so that each $\MC_{\ast,l}(G)$ is indeed a chain complex.
\end{definition}

\begin{remark}
The condition
$
	\ell(x_0,\ldots\widehat{x_i},\ldots,x_k)
	=l
$
appearing in the definition of the differential
can be replaced with the equivalent condition
$
	d(x_{i-1},x_i)+d(x_i,x_{i+1})=d(x_{i-1},x_{i+1}).
$
\end{remark}

\begin{definition}[Magnitude homology]
The \emph{magnitude homology} $\MH_{\ast,\ast}(G)$ of a graph $G$
is the bigraded abelian group defined by
\[
	\MH_{k,l}(G) = \HH_k(\MC_{\ast,l})
\]
for $k,l\geqslant 0$.
\end{definition}

\begin{example}[Complete graphs]
\label{complete:example}
Let $K_n$ denote the complete graph on $n$ vertices.
Then for $l\geqslant 0$, $\MH_{l,l}(K_n)$ is
the free abelian group on $(l+1)$-tuples $(x_0,\ldots,x_l)$
of vertices of $K_n$ satisfying $x_0\neq\cdots\neq x_l$,
and $\MH_{k,l}(K_n) = 0$ if $k\neq l$.
This holds because $\MC_{\ast,\ast}(K_n)$ admits 
exactly the same description, as $d(x_{i},x_{j})=1$ for $i\neq j$, and in particular 
its differentials are all zero.
\end{example}

\begin{example}[Discrete graphs]
\label{discrete:example}
Let $E_n$ denote the discrete graph on $n$ vertices, meaning that it has no edges.
Then $\MH_{0,0}(E_n)$ is the free abelian group on the
vertices of $E_n$, and all other magnitude homology groups
of $E_n$ vanish.  Again, this follows because the magnitude chain
complex admits exactly the same description.
\end{example}

In the two examples above,
the magnitude homology was concentrated on the
diagonal, by which we mean that $\MH_{k,l}(G)=0$
for $k\neq l$.  In Table~\ref{TableFiveCycle} we see that according to computer calculations this does not seem to always be the case,
and we verify this in the example below.

\begin{example}[The cyclic graph $C_5$]
\label{five-cycle:example}
Let $C_5$ denote the cyclic graph with $5$
vertices.  Then $\MH_{2,3}(C_5)$ is isomorphic to the free
abelian group spanned by the oriented edges of $C_5$.
In particular, it is nonzero.

Given an oriented edge $(a_1,a_2)$ of $C_5$,
we will list the vertices of $C_5$ as $a_1,a_2,a_3,a_4,a_5$
by starting at $a_1$, moving to $a_2$, and then
continuing round the graph in the same direction.
Then the generators of $\MC_{2,3}(C_5)$ 
all have one of the four forms
\[
	(a_1,a_2,a_4),
	\quad
	(a_1,a_3,a_4),
	\quad
	(a_1,a_2,a_5),
	\quad
	(a_1,a_3,a_2)
\]
for a uniquely determined oriented edge $(a_1,a_2)$,
and they are all cycles.
Similarly, the generators of $\MC_{3,3}(C_5)$ 
all have one of the four forms
\[
	(a_1,a_2,a_3,a_4),
	\quad
	(a_1,a_2,a_3,a_2),
	\quad
	(a_1,a_2,a_1,a_2),
	\quad
	(a_1,a_2,a_1,a_5)
\]
for a uniquely determined oriented edge $(a_1,a_2)$,
and their boundaries are as follows.
\begin{align*}
\partial(a_1,a_2,a_3,a_4)&=(a_1,a_2,a_4)-(a_1,a_3,a_4)
\\
\partial(a_1,a_2,a_3,a_2)&=(a_1,a_3,a_2)
\\
\partial(a_1,a_2,a_1,a_2)&= 0
\\
\partial(a_1,a_2,a_1,a_5)&= (a_1,a_2,a_5)
\end{align*}
Thus $\MH_{2,3}(C_5)$ is freely generated by the homology
classes of the generators $(a_1,a_2,a_4)$, one for each
oriented edge of $C_5$.
\end{example}

There is an alternative approach to defining magnitude
homology that makes use of simplicial sets and filtered
objects.   We have chosen not to use that approach here in order
to make as few technical requirements of the reader as possible,
but it is discussed later in the paper.
In Section~\ref{KunnethProof:Section}
we explain how the magnitude chain complex $\MC_{\ast,l}(G)$ can be
regarded as the normalised, reduced chain complex
of a pointed simplicial set $M_l(G)$.
In particular, in Remark~\ref{filtration} we explain how there is a filtration
such that each set
$M_l(G)$ arises as a filtration quotient, and in Remark~\ref{spectral} we discuss
the spectral sequence arising from this filtration.

The next result gives the basic relationship
between magnitude homology and magnitude.
It is analogous to~\cite[Proposition~27]{Khovanov}
and~\cite[Theorem~13.3]{OzsvathSzabo}, which state that the graded Euler
characteristic of the Khovanov homology and Heegaard-Floer homology of
a link are the Jones and Alexander polynomial, respectively.
It is our justification for calling magnitude homology
a categorification of magnitude.

\begin{theorem}\label{euler:theorem}
Let $G$ be a graph.  Then
\[
	\sum_{k,l\geqslant 0} (-1)^k\cdot \mathrm{rank}(\MH_{k,l}(G))\cdot  q^l
	=
	\# G.
\]
\end{theorem}

\begin{proof}
Let $\chi$ denote the ordinary Euler characteristic of chain complexes.
Then
\begin{align*}
	\sum_{k,l\geqslant 0} (-1)^k\cdot \mathrm{rank}(\MH_{k,l}(G))\cdot  q^l
	&=
	\sum_{l\geqslant 0} \chi(\MH_{\ast,l}(G))\cdot  q^l
	\\
	&=
	\sum_{l\geqslant 0} \chi(\MC_{\ast,l}(G))\cdot  q^l
	\\
	&=
	\sum_{k,l\geqslant 0} (-1)^k\cdot \mathrm{rank}(\MC_{k,l}(G))\cdot  q^l	
	\\
	&=
	\sum_{k\geqslant 0}
	(-1)^k
	\sum_{x_0\neq\cdots\neq x_k}
	q^{d(x_0,x_1)+\cdots+d(x_{k-1},x_k)}\\
	&= \#G.
\end{align*}
Here the first and third inequalities are the definition of Euler characteristic
of a graded abelian group, the second is a standard property of the Euler
characteristic, and the fourth follows by counting the generators of
$\MC_{k,l}(G)$.
The final equality now follows by~\cite[Proposition~3.9]{LeinsterGraph}.
\end{proof}

We now see some further basic properties of magnitude, these are illustrated in Table~\ref{TableFiveCycle}.  The first proposition explains that the top two entries are the number of vertices and twice the number of edges.

\begin{proposition}\label{elementary:proposition}
Let $G$ be a graph. 
Then $\MH_{0,0}(G)$ is the free abelian group
on the vertices of $G$ and $\MH_{1,1}(G)$ is the
free abelian group on the oriented edges of $G$.
\end{proposition}

\begin{proof}
The same properties hold trivially for chains,
and all differentials involving the terms 
$\MC_{0,0}(G)$ and $\MC_{1,1}(G)$ are zero
(having zero domain or range), so the properties hold
for homology.
\end{proof}

The next proposition explains why the table is lower triangular and why the diameter of the $5$-cycle being $2$ restricts the non-zero entries to being reasonably close to the diagonal.

\begin{proposition}
Let $G$ be a graph and suppose that $\MH_{k,l}(G)\neq 0$.
Then
\[
	k\leqslant l.
\]
Furthermore, if $G$ has diameter $d$, then
\[
	\frac{l}{d}\leqslant k
\]
and moreover
\[
	\frac{l}{d}< k
\]
if $d>1$ and $l>0$.
\end{proposition}

\begin{proof}
If $\MH_{k,l}(G)\neq 0$ then $\MC_{k,l}(G)\neq 0$,
so there is a tuple $(x_0,\ldots,x_k)$ satisfying
\[
	l=d(x_0,x_1)+\cdots d(x_{k-1},x_k).
\]
Each of the summands is at least $1$, since 
consecutive entries are distinct, 
and this gives the first inequality.  
If $G$ has diameter $d$ then each summand is at most $d$,
and this gives the second inequality. 
For the final part suppose that $d>1$,  $l>0$ and $k=l/d$.
Let $(x_0,\ldots,x_k)$ be a generator of $\MC_{k,l}(G)$,
so that $d(x_{i-1},x_i)=d$ for all $i$.
Since $d(x_0,x_1)\geqslant 2$ there is $y\neq x_0,x_1$ 
such that $d(x_0,y)+d(y,x_1)=d(x_0,x_1)$.  
Then $\partial_1(x_0,y,x_1,\ldots,x_k)=(x_0,x_1,\ldots,x_k)$
while $\partial_i(x_0,y,x_1,\ldots,x_k)=0$ for $i=2,\ldots,k-1$.
Thus $(x_0,\ldots,x_k) = \partial(-(x_0,y,x_1,\ldots,x_k))$.
It follows that $\MH_{k,l}(G)=0$.
\end{proof}

Let us conclude the section by verifying that
the operators $\partial$ satisfy the relation
$\partial\circ\partial=0$.  This is a routine
consequence of inequality~\eqref{inequality:one},
but because similar arguments will appear
several times in the rest of the paper, we give a
detailed proof here.

\begin{lemma}\label{differential:lemma}
For any graph $G$, any $k\geqslant 2$, and any $l\geqslant 0$,
the composite
\[
	\MC_{k,l}(G)
	\xrightarrow{\ \partial\ }
	\MC_{k-1,l}(G)
	\xrightarrow{\ \partial\ }
	\MC_{k-2,l}(G)
\]
vanishes.
\end{lemma}

\begin{proof}
It is sufficient to show that for any generator $(x_0,\ldots,x_k)$
of $\MC_{k,l}(G)$, and any $i,j$ in the range  
$0\leqslant i<j\leqslant k$, we have 
\[
	\partial_i\circ\partial_j(x_0,\ldots,x_k)
	=
	\partial_{j-1}\circ\partial_i(x_0,\ldots,x_k).
\]
Each side of this equation is given by either
$(x_0,\ldots,\widehat{x_i},\ldots,\widehat{x_j},\ldots,x_k)$
or $0$.  The left hand side is nonzero if and only if
\[
	\ell(x_0,\ldots,\widehat{x_i},\ldots,x_k)=l
	\quad\mathrm{and}\quad
	\ell(x_0,\ldots,\widehat{x_i},\ldots,\widehat{x_j},\ldots,x_k)
	=l,	
\]
and inequality~\eqref{inequality:one} tells us that
\[
	\ell(x_0,\ldots,\widehat{x_i},\ldots,\widehat{x_j},\ldots,x_k)
	\leqslant
	\ell(x_0,\ldots,\widehat{x_i},\ldots,x_k)
	\leqslant
	\ell(x_0,\ldots,x_k)
	=l,
\]
so that the left hand side is nonzero if and only if
\[
	\ell(x_0,\ldots,\widehat{x_i},\ldots,\widehat{x_j},\ldots,x_k)=l.
\]
A similar argument shows that the right hand side
is nonzero if and only if the same condition holds.
This completes the proof.
\end{proof}

\section{Induced maps}
\label{SectionInducedMaps}
The magnitude of a graph is an element of a set,
the set of formal power series with integer
coefficients.  The magnitude homology of a
graph, on the other hand, is an object of a
category, the category of bigraded abelian
groups.  This categorification gives us the
opportunity to make magnitude into a functor,
and that is what we will do in this section.

In order to make graphs into the objects of
a category we choose the following
notion of morphism.
Given graphs $G$ and $H$, 
a \emph{map of graphs} $f\colon G\to H$
is a map of vertex sets $f\colon V(G)\to V(H)$
satisfying the condition
\[
	\{x,y\}\in E(G)
	\implies
	\{f(x),f(y)\}\in E(H) \mathrm{\ or\ }f(x)=f(y).
\]
In words, a map of graphs is a map of vertex
sets that preserves or contracts each edge.
And in terms of distance, a map of graphs is
a map of vertex sets for which
$d_H(f(x),f(y))\leqslant d_G(x,y)$ for all vertices $x,y\in V(G)$.  From the metric space perspective these distance non-increasing maps are the correct ones to consider in the context of magnitude.
Observe that if $f\colon G\to H$ is a map of
graphs, then the inequality
\begin{equation}\label{inequality:two}
	\ell(f(x_0),\ldots,f(x_k))
	\leqslant
	\ell(x_0,\ldots,x_k)
\end{equation}
holds for any tuple $(x_0,\ldots,x_k)$
of vertices of $G$.

\begin{definition}[Induced chain maps]
Let $f\colon G\to H$ be a map of graphs.
The \emph{induced chain map}
\[f_\#\colon \MC_{\ast,\ast}(G) \longrightarrow \MC_{\ast,\ast}(H)\]
is defined on generators by
\[
	f_\#(x_0,\ldots,x_k)
	=
	\begin{cases}		(f(x_0),\ldots,f(x_k))
		&
		\mathrm{if\ } \ell(f(x_0),\ldots,f(x_k))=\ell(x_0,\ldots,x_k)
		\\
		0
		&
		\mathrm{otherwise}.	
	\end{cases}
\]
\end{definition}

If $f\colon G\to H$ is a map of graphs then
the relation $f_\#\circ\partial = \partial\circ f_\#$ holds,
so that $f_\#$ is indeed a chain map. 
The proof, which we omit, is similar to that of
Lemma~\ref{differential:lemma},
and makes use of inequalities~\eqref{inequality:one}
and~\eqref{inequality:two}.

\begin{definition}[Induced maps in homology]
Let $f\colon G\to H$ be a map of graphs.
The \emph{induced map in homology} is the map
\[
	f_\ast\colon \MH_{\ast,\ast}(G)\longrightarrow \MH_{\ast,\ast}(H)
\]
induced by $f_\#$.
\end{definition}

\begin{proposition}
The assignment $G\mapsto \MH_{\ast,\ast}(G)$,
$f\mapsto f_\ast$ is a functor from the category of graphs
to the category of bigraded abelian groups.
\end{proposition}

That
the identity map of a graph induces the
identity map in homology is immediate. 
To prove that $g_\ast\circ f_\ast
= (g\circ f)_\ast$ holds for any maps of
graphs $f\colon G\to H$ and $g\colon H\to K$
one proceeds as in
Lemma~\ref{differential:lemma},
making use of both~\eqref{inequality:two}
and~\eqref{inequality:one}.
The details are left to the reader.

Recall from Proposition~\ref{elementary:proposition}
that $\MH_{0,0}(G)$ is the free abelian group on the set of
vertices of $G$, and that $\MH_{1,1}(G)$ is the free
abelian group on the set of oriented edges of $G$.
The following result, whose proof is an immediate
consequence of the definitions, describes the effect
of induced maps in these degrees.

\begin{proposition}
Let $f\colon G\to H$ be a map of graphs.
Then $f_\ast\colon \MH_{0,0}(G)\to \MH_{0,0}(H)$ sends a vertex $x$ to $f(x)$.
And $f_\ast\colon \MH_{1,1}(G)\to \MH_{1,1}(H)$ sends an edge $\{x,y\}$
to $\{f(x),f(y)\}$ if that is an edge, and to $0$ otherwise.
\end{proposition}

\begin{corollary}
Let $f\colon G\to H$ be a map of graphs.
If $f_\ast\colon \MH_{\ast,\ast}(G)
\to \MH_{\ast,\ast}(H)$ is an isomorphism,
then $f$ is an isomorphism of graphs.
\end{corollary}

\section{Disjoint unions}
\label{SectionDisjointUnions}
In this brief section we proof the additivity of
magnitude homology with respect to disjoint
unions.   As an immediate corollary we get the additivity of the magnitude.

\begin{proposition}\label{disjoint}
Let $G$ and $H$ be graphs and write $i\colon G\to G\sqcup H$
and $j\colon H\to G\sqcup H$ for the inclusion maps.
Then the induced map
\[
	i_\ast\oplus j_\ast\colon \MH_{\ast,\ast}(G)\oplus \MH_{\ast,\ast}(H)
	\longrightarrow \MH_{\ast,\ast}(G\sqcup H)
\]
is an isomorphism.
\end{proposition}

\begin{proof}
Let $(x_0,\ldots,x_k)$ be a generator
of $\MC_{k,l}(G\sqcup H)$. 
Since $\ell(x_0,\ldots,x_k)=l$, 
we have $d(x_{i-1},x_i)<\infty$ for all $i$, 
so that $x_0,\ldots,x_k$ all belong to $G$ or all
belong to $H$. Consequently $i_\#\oplus j_\#$
is an isomorphism, and the result follows.
\end{proof}

\begin{corollary}[{Leinster~\cite[Lemma~3.5]{LeinsterGraph}}]
Let $G$ and $H$ be graphs.
Then $\#(G\sqcup H) = \#G + \#H$.
\end{corollary}
\begin{proof}
This follows from Proposition~\ref{disjoint},
Theorem~\ref{euler:theorem}, and the fact that 
$\chi(C_\ast\oplus D_\ast) = 
\chi(C_\ast) + \chi(D_\ast)$
for finitely generated graded abelian groups
$C_\ast$ and $D_\ast$.
\end{proof}

\section{Cartesian products}
\label{SectionCartesianProducts}
In this section we state a K\"unneth Theorem for magnitude homology with respect to the cartesian product of graphs and we give an example.  The proof of the theorem is given is Section~\ref{KunnethProof:Section}.

The \emph{cartesian product} $G\cp H$
of graphs $G$ and $H$ has vertex 
set $V(G)\times V(H)$, and has 
an edge from $(x_1,y_1)$ to $(x_2,y_2)$
if either $x_1=x_2$ and $\{y_1,y_2\}$ is an edge in $H$,
or $y_1=y_2$ and $\{x_1,x_2\}$ is an edge in $G$.
The metric on $G\cp H$ is given by
\begin{equation}\label{cpmetric}
	d_{G\cp H}((x_1,y_1),(x_2,y_2)) = d_G(x_1,x_2) + d_H(y_1,y_2)
\end{equation}
for $(x_1,y_1),(x_2,y_2)\in V(G\cp H)$.

\begin{remark}The cartesian product is not the \emph{categorical} product on the category of graphs, but Equation~\eqref{cpmetric} tells us that it is the natural tensor product from the perspective of enriched category theory.
\end{remark}
\begin{definition}[The exterior product]\label{exterior}
Fix $l\geqslant 0$.
The \emph{exterior product} is the map
\begin{equation}\label{exterior-chains}
	\cp\colon \MC_{\ast,\ast}(G)\otimes \MC_{\ast,\ast}(H)
	\longrightarrow
	\MC_{\ast,\ast}(G\cp H)
\end{equation}
whose component
\[
	\cp\colon \MC_{k_1,l_1}(G)\otimes \MC_{k_2,l_2}(H)
	\longrightarrow
	\MC_{k,l}(G\cp H)
\]
for $k_1,k_2\geqslant 0$ with $k_1+k_2=k$
is defined by
\[
	(x_0,\ldots,x_{k_1})\cp(y_0,\ldots,y_{k_2})
	=
	\sum_\sigma \mathrm{sign}(\sigma)\cdot
	((x_{i_0},y_{j_0}),\ldots,(x_{i_{k}},y_{j_k})).
\]
Here the sum ranges over all sequences
$\sigma=((i_0,j_0),\ldots,(i_k,j_k))$ 
for which
$i_0=j_0=0$, for which $0\leqslant i_r\leqslant k_1$
and $0\leqslant j_r\leqslant k_2$ for all $r$,
and for which each term $(i_{r+1},j_{r+1})$ is
obtained from its predecessor $(i_r,j_r)$ by increasing exactly one
of the two components by $1$.
Given such a sequence, we define $\mathrm{sign}(\sigma)=(-1)^n$
where $n$ is the number of pairs $(i,j)$ for which $i=i_r \implies j<j_r$.
Compare with the discussion in~\cite[pp.~277-278]{Hatcher}.
The exterior product is a chain map, and so induces a map
in homology that we indicate by the same symbol,
\begin{equation}\label{exterior-homology}
	\cp\colon \MH_{\ast,\ast}(G)\otimes \MH_{\ast,\ast}(H)
	\longrightarrow
	\MH_{\ast,\ast}(G\cp H).
\end{equation}
\end{definition}

\begin{theorem}[The K\"unneth theorem for magnitude homology]
\label{kunneth:theorem}
The exterior product in homology~\eqref{exterior-homology}
fits into a natural short exact sequence
	\begin{multline*}
		0
		\longrightarrow
		\MH_{\ast,\ast}(G)\otimes \MH_{\ast,\ast}(H)
		\xrightarrow{\ \cp\ }
		\MH_{\ast,\ast}(G\cp H)
		\\
		\longrightarrow
		\mathrm{Tor}(\MH_{\ast-1,\ast}(G), \MH_{\ast,\ast}(H)
		\longrightarrow
		0
	\end{multline*}
that is non-naturally split.  In particular, $\cp$ becomes an
isomorphism after tensoring with the rationals, and is an
isomorphism if either $\MH_{\ast,\ast}(G)$ or $\MH_{\ast,\ast}(H)$
is torsion-free.
\end{theorem}

The proof of this theorem is deferred 
to Section~\ref{KunnethProof:Section}.

\begin{example}[The cyclic graph $C_4$]
The magnitude homology of the cyclic graph $C_4$ with four 
vertices is 
\[
	\MH_{k,l}(C_4)
	=
	\left\{
	\begin{array}{ll}
		\Z^{4(l+1)} & \text{if }k=l, \\
		0 & \text{otherwise.} 
	\end{array}
	\right.
\]
To see this, observe that $C_4=K_2\cp K_2$.
Example~\ref{complete:example} shows that
$MH_{k,l}(K_2)$ vanishes if $k\neq l$ and that it is free abelian
on two generators if $k=l$.
Since these groups contain no torsion the
K\"unneth Theorem (Theorem~\ref{kunneth:theorem}) shows that
\[
	\MH_{\ast,\ast}(C_4)
	\cong
	\MH_{\ast,\ast}(K_2)\otimes\MH_{\ast,\ast}(K_2),
\]
or more explicitly that
\[
	\MH_{k,l}(C_4)
	\cong
	\bigoplus_{
		\begin{array}{c}
			\scriptstyle k_1+k_2=k \\ 
			\scriptstyle l_1+l_2=l
		\end{array}
	}
	\MH_{k_1,l_1}(K_2)\otimes \MH_{k_2,l_2}(K_2),
\] 
and the claim follows.
\end{example}

\begin{remark}
We know of no graph $G$ for which $\MH_{\ast,\ast}(G)$ contains
torsion.
\end{remark}

\section{The Mayer-Vietoris sequence}
\label{SectionMayerVietoris}
In this section we show that a Mayer-Vietoris theorem holds for so-called projecting decompositions of graphs.  
The long exact Mayer-Vietoris sequence actually breaks up into split short exact sequences.  From this we will obtain Leinster's inclusion-exclusion principle.
The proof is given in Section~\ref{SectionExcision}.

We begin by recalling some definitions of Leinster.  
Firstly, convexity is supposed to be reminiscent of the idea that in convex subset of $\mathbb{R}^{n}$ each pair of points is connected by a geodesic which is also contained in the subset.
\begin{definition}[Convex {\cite[Definition~4.2]{LeinsterGraph}}] 
A subgraph $U\subset X$ is called \emph{convex}
if $d_U(u,v)=d_X(u,v)$ for all $u,v\in U$.
\end{definition}
Secondly, projecting to a convex subgraph is reminiscent of the idea that there is a `nearest point map'  from $\mathbb{R}^{n}$ to any convex subset.
\begin{definition}[Projecting~{\cite[Definition~4.6]{LeinsterGraph}}] 
Let $U\subset X$ be a convex subgraph.  We say that 
$X$ \emph{projects} to $U$ if for every $x\in X$
that can be connected by an edge-path to some vertex of $U$,
there is $\pi(x)\in U$ such that for all $u\in U$ we have
\[
	d(x,u) = d(x,\pi(x)) + d(\pi(x),u).
\]
Thus $\pi(x)$ is the unique point of $U$ closest to $x$.
Writing $X_U$ for the component of $X$ consisting
of vertices that admit an edge path to $U$, there is then a map
$\pi\colon X_U\to U$ defined by $u\mapsto \pi(u)$.
\end{definition}
Every even cyclic graph projects to any of its edges, whereas no odd graph projects to any of its edges.  Projecting to $U$ is stronger than each point having a closest point in $U$ as can be seen by considering two adjacent edges as a subgraph of the $5$-cycle graph.

Suppose that $X$ is a graph that is the union of subgraphs $G$ and $H$,
such that $G\cap H$ is convex in $G\cup H$,
and such that $H$ projects to $G\cap H$.
Leinster~\cite[Theorem~4.9]{LeinsterGraph} has shown
that in this situation the magnitude satisfies
the inclusion-exclusion principle
$\# X = \#G+\#H-\#(G\cap H)$.  We will categorify this to a
Mayer-Vietoris sequence relating the magnitude
homologies of $G\cap H$, $G$, $H$ and $G\cup H$.

\begin{definition}[Projecting decompositions]
A \emph{projecting decomposition} is a triple $(X;G,H)$ consisting
of a graph $X$ and subgraphs $G$ and $H$ such that
the following properties hold.
\begin{itemize}
	\item
	$X=G\cup H$
	\item
	$G\cap H$ is convex in $X$
	\item
	$H$ projects to $G\cap H$
\end{itemize}
Given a projecting decomposition $(X;G,H)$, we write
\[
	i^G\colon G\to X,
	\quad
	i^H\colon H\to X,
	\quad
	j^G\colon G\cap H\to G,
	\quad
	j^H\colon G\cap H\to H
\]
for the inclusion maps.
A \emph{decomposition map} $f\colon (X;G,H)\to (X';G',H')$ is
a map of graphs $f\colon X\to X'$ such that $f(G)\subset G'$
and $f(H)\subset H'$.  It is a \emph{projecting} decomposition map if
$H_{G\cap H}=f^{-1}(H'_{G'\cap H'})$ 
and $f(\pi(h))=\pi(f(h))$ for all $h\in H_{G\cap H}$.
\end{definition}

\begin{definition}
Given a projecting decomposition $(X;G,H)$,
 let $\MC_{\ast,\ast}(G,H)$ 
denote the subcomplex of $\MC_{\ast,\ast}(G\cup H)$ 
spanned by those tuples $(x_0,\ldots,x_k)$
whose entries all lie in $G$ or all lie in $H$.
\end{definition}

\begin{theorem}[Excision for magnitude chains]
\label{excision:theorem}
Let $(X;G,H)$ be a projecting decomposition.
For all $l\geqslant 0$ the inclusion
$\MC_{\ast,l}(G,H)\hookrightarrow \MC_{\ast,l}(G\cup H)$
is a quasi-isomorphism.
\end{theorem}

This result is a version of excision for
the magnitude chain complex, and is closely
analogous to versions of excision that hold
for the singular chain complex of a topological
space, see for example~\cite[Proposition~2.21]{Hatcher}
or~\cite[Proposition~7.3]{Dold}. The proof of Theorem~\ref{excision:theorem}
is deferred until Section~\ref{SectionExcision}.  From excision we get the Mayer-Vietoris Theorem.

\begin{theorem}[Mayer-Vietoris for magnitude homology]
\label{MV:Theorem}
Let $(X;G,H)$ be a projecting decomposition.  Then there is a split short exact sequence:
\begin{multline*}
	0
	\to
	\MH_{\ast,\ast}(G\cap H)
	\xrightarrow{(j^G_\ast,-j^H_\ast)}
	\MH_{\ast,\ast}(G)\oplus \MH_{\ast,\ast}(H)\\
	\xrightarrow{i^G_\ast\oplus i^H_\ast}
	\MH_{\ast,\ast}(G\cup H)
	\to
	0.
\end{multline*}
The sequence is natural with respect to decomposition maps,
and the splitting is natural with respect to projecting decomposition maps.
\end{theorem}

Theorem~\ref{MV:Theorem} is stronger than one might anticipate,
since it gives a short exact sequence in each homological degree,
rather than the single long exact sequence that is familiar from
the Mayer-Vietoris theorem for singular homology~\cite[p.~149]{Hatcher}.
In fact, our short exact sequence is obtained from
a long exact sequence of Mayer-Vietoris type by splitting
it into the short exact sequences.
The splitting is possible due to our
assumption that in a projecting decomposition $(X;G,H)$ the subgraph $H$
projects onto $G\cap H$.  This assumption is impossible
to remove, as shown in Section~\ref{section:projectingnecessary}.
The proof of Theorem~\ref{MV:Theorem} is also deferred until Section~\ref{SectionExcision}.

\begin{corollary}[Inclusion-exclusion {\cite[Theorem~4.9]{LeinsterGraph}}]
If $(X;G,H)$ is a projecting decomposition then
$\# X = \#G+\#H-\#(G\cap H)$.
\end{corollary}

\begin{proof}
From the short exact sequence in Theorem~\ref{MV:Theorem}
it follows that 
\[
	\chi(\MH_{\ast,l}(G\cap H))
	-\chi(\MH_{\ast,l}(G)\oplus \MH_{\ast,l}(H))
	+\chi(\MH_{\ast,l}(G\cup H))=0.
\]
Since Euler characteristic is additive with
respect to direct sums, this rearranges to
give
\[
	\chi(\MH_{\ast,l}(G\cup H))
	=
	\chi(\MH_{\ast,l}(G)) + \chi(\MH_{\ast,l}(H))
	-\chi(\MH_{\ast,l}(G\cap H)).
\]
Multiplying this equation by $q^l$, then summing
over all $l\geqslant 0$ and applying Theorem~\ref{euler:theorem},
the claim follows.
\end{proof}

\begin{corollary}[Magnitude homology of trees]
\label{tree:corollary}
Let $T$ be a tree.  Then
\[
	\MH_{k,l}(T)
	\cong
	\begin{cases}
		\Z V(T) & \text{if }k=l=0,\\
		\Z \vec{E}(T) & \text{if }k=l>0,\\
		0 & \text{if }k\neq l.	
	\end{cases}
\]
This isomorphism is natural with respect to maps of trees,
where $\Z\vec{E}(T)$ is made into a functor of $T$ by
declaring that if $f\colon T\to S$ is a map, then
$f_\ast\colon\vec{E}(T)\to\vec{E}(S)$ sends an oriented edge
$(x,y)$ to $(f(x),f(y))$ if $f(x)\neq f(y)$, and to $0$ if $f(x)=f(y)$.
\end{corollary}

\begin{proof}
Let us write $F_{k,l}$ for the functor appearing on the right hand side
of the desired isomorphism.
There is a natural transformation $\theta\colon F_{k,l}\Rightarrow\MH_{k,l}$
given on generators by $\theta_T(x)=(x)$ if $k=l=0$ and $x\in V(T)$,
and by
\[
	\theta_T((x,y))
	=
	(\underbrace{x,y,x,y,\ldots}_{k+1})
\]
if $k=l>0$ and $(x,y)\in\vec{E}(T)$.
We prove that $\theta_T$ is an isomorphism by induction
on the number of edges of $T$.
Observe that $\theta_T$ is trivially an isomorphism if $T$
has no edges or a single edge.
In general, if $T$ has two or more edges then we may write
$T=T_1\cup T_2$ where $T_1$, $T_2$ and $T_1\cap T_2$ are
subtrees of $T$.  It is immediate that $T_1\cap T_2$ is
convex in $T$ and that $T_2$ projects to $T_1\cap T_2$,
so that by Theorem~\ref{MV:Theorem} we have a short exact sequence
\[
	0
	\to
	\MH_{k,l}(T_1\cap T_2)
	\to
	\MH_{k,l}(T_1)\oplus\MH_{k,l}(T_2)
	\to
	\MH_{k,l}(T)
	\to
	0.
\]
There is an analogous short exact sequence 
in which $\MH_{k,l}$ is replaced with $F_{k,l}$, and it can be combined
with the one above to form the rows of a commuting diagram
whose vertical arrows are obtained using $\theta$.
The first two vertical arrows, $\theta_{T_1\cap T_2}$ and
$\theta_{T_1}\oplus\theta_{T_2}$, are isomorphisms by induction,
and it follows that $\theta_T$ is an isomorphism as well.
\end{proof}

\begin{corollary}[Wedge sums]
Let $G$ and $H$ be graphs with chosen base vertices,
and let $G\vee H$ denote the graph obtained by identifying
the two base vertices to a single vertex $P$.
Then the inclusion maps
$a\colon G\to G\vee H$ and $b\colon H\to G\vee H$
induce an isomorphism
\[
	a_\ast\oplus b_\ast
	\colon
	\MH_{\ast,\ast}(G)\oplus_{\MH_{\ast,\ast}(P)} \MH_{\ast,\ast}(H)
	\xrightarrow{\ \cong\ }
	\MH_{\ast,\ast}(G\vee H).
\]
\end{corollary}

\begin{proof}
By considering $G\vee H$ as the union of
$G$ and $H$, and observing that $H$ projects
onto $P$, we obtain a projecting decomposition $(G\vee H;G,H)$.
The result then follows from Theorem~\ref{MV:Theorem}.
\end{proof}

\section{Diagonal graphs}
\label{SectionDiagonalGraphs}
We have seen that complete graphs (Example~\ref{complete:example}), 
discrete graphs (Example~\ref{discrete:example}),
and trees (Corollary~\ref{tree:corollary}) are all diagonal  in the following sense.
\begin{definition}[Diagonality]
A graph is \emph{diagonal} if its magnitude homology
is concentrated on the diagonal.
In other words $G$ is diagonal if $\MH_{k,l}(G)=0$ for $k\neq l$.
\end{definition}

We have also seen that the five-cycle $C_5$
is not diagonal (Example~\ref{five-cycle:example}).
In this section we will give some rather general results
that demonstrate diagonality in various situations.
These will be enough for us to explain in general terms
all the instances of diagonality that we have seen so far,
and also several more, including complete multipartite
graphs and the octahedral graph.  The magnitude
of a diagonal graph has coefficients that alternate in sign;
our examples of diagonal graphs explain all instances of 
this phenomenon that are known so far.

\begin{proposition}[Diagonal graphs and magnitude]
\label{diagonal-magnitude:proposition}
If $G$ is diagonal then the coefficients of the magnitude
$\#G$ alternate in sign, and $\#G$ determines $\MH_{\ast,\ast}(G)$
up to isomorphism.
\end{proposition}

\begin{proof}
Let $G$ be diagonal.
By Theorem~\ref{euler:theorem} we have
\[
	\# G
	=
	\sum_{l\geqslant 0} (-1)^l\cdot \mathrm{rank}(\MH_{l,l}(G))\cdot  q^l
\]
so that the coefficients alternate in sign as claimed, and 
$\# G$ determines the quantities $\mathrm{rank}(\MH_{l,l}(G))$.
Since the chain groups $\MC_{l+1,l}(G)$ are identically $0$,
it follows that the $\MH_{l,l}(G)$ are free abelian, and so
determined by their ranks.  This completes the proof.
\end{proof}

\begin{proposition}
The cartesian product of diagonal graphs is diagonal.
A graph that admits a projecting decomposition into
diagonal graphs is itself diagonal.
\end{proposition}

\begin{proof}
The first claim follows from Theorem~\ref{kunneth:theorem},
together with the fact that if $G$ is diagonal then each group $\MH_{l,l}(G)$
is torsion-free (see the proof of 
Proposition~\ref{diagonal-magnitude:proposition}).
The second follows from Theorem~\ref{MV:Theorem}.
\end{proof}

\begin{definition}
Let $G$ and $H$ be graphs.  The \emph{join} of $G$ and $H$,
denoted $G\join H$, is the graph obtained from $G\sqcup H$
by adding the edges $\{x,y\}$ for all $x\in V(G)$ and $y\in V(H)$.
\end{definition}

\begin{theorem}\label{join:theorem}
Let $G$ and $H$ be graphs such that $G,H\neq \emptyset$.
Then the join $G\join H$ is diagonal.
\end{theorem}

The proof, which is rather lengthy, 
is deferred until Section~\ref{join:proof:section}.

\begin{example}[Complete multipartite graphs]
The complete multipartite graph with maximal independent subsets
of size $n_1,\ldots,n_k$ is the iterated join
$E_{n_1}\join\cdots \join E_{n_k}$, and so is diagonal
by Theorem~\ref{join:theorem}.
This also gives another proof that complete graphs
are diagonal, since they are iterated joins of copies of $E_1$,
and that $C_4$ is diagonal, since it is $E_2\join E_2$.
\end{example}

\begin{example}[1-skeleta of platonic solids]
The $1$-skeleta of the tetrahedron, cube and octahedron
are all diagonal, since they are $K_4$, $K_2\cp K_2\cp K_2$
and $E_2\star E_2\star E_2$ respectively.
The $1$-skeleton of the dodecahedron is not diagonal:
its magnitude homology in bidegree $(2,3)$ is nonzero,
as one sees by adapting the reasoning of 
Example~\ref{five-cycle:example}.
On the other hand it appears from Sage computations (see Appendix~\ref{section:icosahedron})
that the $1$-skeleton of the icosahedron is diagonal,
though we cannot prove it using the techniques of this section.
\end{example}

\section{Proof of the K\"unneth Theorem}
\label{KunnethProof:Section}

We now give the proof of Theorem~\ref{kunneth:theorem}.
While the proofs in the previous sections were complicated
but not strictly speaking technical, the present proof
is indeed technical, relying on the version of the 
K\"unneth theorem that applies to the homology of simplicial sets.

\begin{definition}[The simplicial set $M_l(G)$]
Let $G$ be a graph and let $l\geqslant 0$.  
We define $M_l(G)$ to be the {pointed} simplicial set
whose $k$-simplices are the $(k+1)$-tuples
$(x_0,\ldots,x_k)$ of length $l$, plus a basepoint simplex.
(Adjacent entries are allowed to be equal.)
The $i$-th face map deletes the $i$-th entry of a tuple if this preserves
the length, and sends it to the basepoint otherwise.
The $i$-th degeneracy doubles the $i$-th entry of a tuple.
The faces and degeneracies all send basepoints to basepoints.
\end{definition}

Observe that the non-degenerate, non-basepoint $k$-simplices
of $M_l(G)$ are precisely the generators of $\MC_{k,l}(G)$.

\begin{proposition}[A simplicial K\"unneth theorem]
Let $G$ and $H$ be graphs and fix $l\geqslant 0$.
Then the map of pointed simplicial sets
\[
	\cp\colon\bigvee_{l_1+l_2=l}M_{l_1}(G)\wedge M_{l_2}(H)
	\longrightarrow
	M_l(G\cp H)
\]
defined by 
\[
	(x_0,\ldots,x_k)\cp(y_0,\ldots,y_k)
	=
	((x_0,y_0),\ldots,(x_k,y_k))
\]
is an isomorphism.
\end{proposition}

\begin{proof}
That the map is simplicial and an isomorphism both
follow from the observation that
\[
	\ell((x_0,y_0),\ldots,(x_k,y_k))
	=\ell(x_0,\ldots,x_k)+\ell(y_0,\ldots,y_k)
\]
for any tuple $((x_0,y_0),\ldots,(x_k,y_k))$ of vertices of $G\cp H$.
\end{proof}

Given a pointed simplicial set $X$, we write $\bar N_\ast (X) $ for
the \emph{normalised reduced chain complex} of $X$.
This is given in degree $k$ by the free abelian group 
on $X_k$, divided out by the span of the degenerate simplices
and the basepoint.  The differential
$d\colon \bar N_k(X)\to\bar N_{k-1}(X)$ is given
by $d = \sum_{i=0}^k (-1)^i d_i$, where $d_i$ denotes the $i$-th
face map (extended linearly).
The following is immediate from the definitions.

\begin{lemma}
$\MC_{\ast,l}(G) = \bar N_\ast(M_l(G))$.
\end{lemma}

Given pointed simplicial sets $X$ and $Y$, we define
the \emph{reduced normalised Eilenberg-Zilber map}
\[
	\nabla^{\bar N}\colon \bar N_\ast(X)\otimes\bar N_\ast(Y)
	\longrightarrow
	\bar N_\ast(X\wedge Y)
\]
by
\[
	\nabla^{\bar N} (x\otimes y)
	=
	\sum_\sigma\mathrm{sign}(\sigma)(c\circ (x\times y)\circ\sigma)
\]
for $x\in X_p$ and $y\in Y_q$ non-degenerate, non-basepoint.
Here  $c\colon X\times Y\to X\wedge Y$ denotes the collapse map while $\sigma$ and $\mathrm{sign}(\sigma)$ are like those in Definition~\ref{exterior}
except we are regarding $x$, $y$ and $\sigma$ as simplicial maps
$x\colon\Delta[p]\to X$, $y\colon\Delta[q]\to Y$ and
$\sigma\colon\Delta[p+q]\to\Delta[p]\times\Delta[q]$,
so that $c\circ (x\times y)\circ\sigma$ is a simplicial map
$\Delta[p+q]\to X\wedge Y$, or in other words an element of 
$(X\wedge Y)_{p+q}$.
The following fact is presumably standard, but we do not know of
a proof that applies to \emph{reduced} normalised chains.

\begin{proposition}
The Eilenberg-Zilber map $\nabla^{\bar N}$ is a quasi-isomorphism.
\end{proposition}

\begin{proof}
Given a simplicial set $Z$, let us write 
$N_\ast(Z)$ for the normalised chains on $Z$, or in other words
the standard chains on $Z$ divided out by the span of the degenerate
elements.  See section~4 of~\cite{EM}.
Let $U$ and $V$ be simplicial sets.
As in section~5 of~\cite{EM}, the standard Eilenberg-Zilber map
reduces to a map
\[
	\nabla^N\colon
	N_\ast(U)\otimes N_\ast(V)
	\longrightarrow
	N_\ast(U\times V)
\]
that is a chain homotopy equivalence.
The definition of the Eilenberg-Zilber map is given in line (5.3)
of~\cite{EM}, and it is simple to use this to verify that
\[
	\nabla^N(u\otimes v) = 
	\sum_\sigma\mathrm{sign}(\sigma)(u\times v)\circ\sigma
\]
for $u\in U_p$ and $v\in V_q$ non-degenerate, with the right-hand-side
understood as in the definition of $\nabla^{\bar N}$.
One sees that the diagram below commutes.
\[\xymatrix{
	N_\ast(X)\otimes N_\ast(Y)
	\ar[r]^-{\nabla^N}
	\ar[d]
	&
	N_\ast(X\times Y)
	\ar[d]
	\\
	\bar N_\ast(X)\otimes \bar N_\ast(Y)
	\ar[r]^-{\nabla^{\bar N}}
	&
	\bar N_\ast(X\wedge Y)
}\]
The kernels of the vertical maps are 
$N_\ast(X)\otimes N_0(\pt)+N_0(\pt)\otimes N_\ast(Y)$
and
$N_\ast(X\vee Y)$
respectively, and it is evident from the formula that $\nabla^N$
restricts to an isomorphism between these.  Since $\nabla^N$ is 
a quasi-isomorphism, it follows that $\nabla^{\bar N}$ is as well. 
\end{proof}

\begin{proof}[Proof of Theorem~\ref{kunneth:theorem}]
The composite
\begin{align*}
	\bigoplus_{l_1+l_2=l}\MC_{\ast,l_1}(G)\otimes \MC_{\ast,l_2}(H)
	&\xrightarrow{\ \ \ =\ \ \ }
	\bigoplus_{l_1+l_2=l}\bar N(M_{l_1}(G))\otimes\bar N(M_{l_2}(H))
	\\
	&\xrightarrow{\bigoplus\nabla^{\bar N}\ }
	\bigoplus_{l_1+l_2=l}\bar N(M_{l_1}(G)\wedge M_{l_2}(H))
	\\
	&\xrightarrow{\ \ \ =\ \ \ }
	\bar N\left(\bigvee_{l_1+l+2=l}M_{l_1}(G)\wedge M_{l_2}(H)\right)
	\\
	&\xrightarrow{\ \bar N(\cp)\ }
	\bar N(M_{l}(G\cp H))
	\\
	&\xrightarrow{\ \ \ =\ \ \ }
	\MC_{\ast,l}(G\cp H)	
\end{align*}
consists of isomorphisms and one quasi-isomorphism,
and so is itself a quasi-isomorphism.
Unravelling the definitions shows that this composite is precisely
the map $\cp$.
The result then follows by applying the Algebraic K\"unneth Theorem~\cite[Theorem~3B.5]{Hatcher}.
\end{proof}

\begin{remark}[$M_l(G)$ as filtration quotients]
\label{filtration}
We may realise each $M_l(G)$ as a filtration quotient
of a filtered simplicial set, as follows.
Define $\MS(G)$ to be the simplicial set whose $k$-simplices are 
finite-length $(k+1)$-tuples $(x_0,\ldots,x_k)$ of vertices of $G$,
in which the $i$-th face map deletes the $i$-th entry,
and in which the $i$-th degeneracy doubles the $i$-th entry.
Form the filtration
\[
	\MS_0(G)\subset \MS_1(G)\subset \MS_2(G)\subset\cdots	
\]
of $\MS(G)$ in which $\MS_l(G)$ consists of all tuples of length at most $l$.
Then $M_l(G) = \MS_l(G)/ \MS_{l-1}(G)$.
The simplicial set $\MS(G)$ has one component for each component of $G$,
and it is not difficult to show that each component is contractible. 
\end{remark}

\begin{remark}[The simplicial approach to magnitude homology]
\label{simplicialapproach}
Readers with the relevant background
in abstract homotopy theory may find it more natural to think
about magnitude homology using the pointed simplicial sets
$M_l(G)$ introduced in this section, and indeed using the filtered
simplicial set $\MS(G)$ of the previous remark, rather than using the
definition of $\MC_{\ast,\ast}(G)$.  
We have chosen to downplay this simplicial approach in order to
make the paper as accessible as possible,
in particular to readers coming from graph theory and category theory.
We expect that large parts of our work could be
`lifted' to the context of simplicial sets,
however it is not clear that this would lead to any
significant simplifications in the material covered here.
Moreover, considering how little we know of magnitude homology
(for example, we do not know any graphs whose magnitude
homology contains torsion), it seems reasonable
to limit ourselves to a homological approach at this stage.
\end{remark}

\begin{remark}[A spectral sequence]
\label{spectral}
The filtered simplicial set $\MS(G)$ of the previous two remarks
gives rise to a spectral sequence $(E^r_{\ast,\ast})_{r\geqslant 1}$
whose $E^1$-page is obtained from the homology of the filtration
quotients of $\MS(G)$, and which converges to the homology of $\MS(G)$.
To be precise, $E^1_{i,j} = \MH_{i+j,i}(G)$,
$E^\infty_{i,j}=0$ for $(i,j)\neq (0,0)$, and $E^\infty_{0,0}=\Z^c$
where $c$ denotes the number of components of $G$.
\end{remark}

\section{Proof of excision and Mayer-Vietoris}
\label{SectionExcision}

In this section we give the proofs of 
Theorem~\ref{excision:theorem} (excision for magnitude chains)
and Theorem~\ref{MV:Theorem} (Mayer-Vietoris for magnitude
homology).
To that end we fix throughout the section a projecting
decomposition $(X;G,H)$.
In this situation the pairs
\[
	G\cap H\subset X
	\qquad
	G\cap H\subset G
	\qquad
	G\cap H \subset H
	\qquad
	G\subset X
	\qquad
	H\subset X
\]
are all convex.  In the first case this is an assumption.
In the second and third cases it is an immediate consequence of the first.
And in the fourth and fifth cases it follows 
from~\cite[Lemma~4.3]{LeinsterGraph}.
Thus the length of a tuple $(x_0,\ldots,x_k)$ of vertices of $X$
is unambiguously defined:  even if the vertices happen to all lie
in $G$ or $H$ or $G\cap H$, the length is the same whichever 
graph one regards the tuple as belonging to.

\begin{proof}[Proof of Theorem~\ref{MV:Theorem}, 
assuming Theorem~\ref{excision:theorem}]
Fix $l\geqslant 0$. It follows from our remarks on lengths of
tuples in $X$ that the sequence 
\[
	0
	\to
	\MC_{\ast,l}(G\cap H)
	\xrightarrow{(j^G_\#,-j^H_\#)}
	\MC_{\ast,l}(G)\oplus \MC_{\ast,l}(H)
	\longrightarrow
	\MC_{\ast,l}(G,H)
	\to
	0
\]
is short exact.
Taking the associated long exact sequence
and using the isomorphism 
$H_\ast(\MC_{\ast,l}(G,H))\cong\MH_{\ast,l}(G\cup H)$
induced by the quasi-isomorphism of Theorem~\ref{excision:theorem},
one obtains the following long exact sequence.
\begin{multline*}
	\cdots
	\to
	\MH_{\ast,\ast}(G\cap H)
	\xrightarrow{(j^G_\ast,-j^H_\ast)}
	\MH_{\ast,\ast}(G)\oplus \MH_{\ast,\ast}(H)
	\\
	\xrightarrow{i^G_\ast + i^H_\ast}
	\MH_{\ast,\ast}(G\cup H)
	\xrightarrow{\partial}
	\MH_{\ast-1,\ast}(G\cap H)
	\to\cdots
\end{multline*}
Writing $H=A\sqcup B$ where $A$ is the full subgraph consisting
of vertices that can be joined to $G\cap H$ by an edge-path,
and $\pi\colon A\to G\cap H$ for the projection map,
it follows that the composite
\begin{align*}
	\MH_{\ast,\ast}(G)\oplus\MH_{\ast,\ast}(H)
	&\xrightarrow{\phantom{-\pi_\ast}\mathllap{=\ }}
	\MH_{\ast,\ast}(G)\oplus \MH_{\ast,\ast}(A)\oplus\MH_{\ast,\ast}(B)
	\\
	&\xrightarrow{\phantom{-\pi_\ast}}
	\MH_{\ast,\ast}(A)
	\\
	&\xrightarrow{-\pi_\ast}
	\MH_{\ast,\ast}(G\cap H)
\end{align*}
is left inverse to $(j^G_\ast,-j^H_\ast)$.
Consequently the long exact sequence splits into 
the split short exact sequences of the statement.
The naturality claims are immediately verified.
\end{proof}

Now we move on to the proof 
of Theorem~\ref{excision:theorem}, which is our excision 
theorem for magnitude chains.
While the statement of our theorem
is closely analogous
to versions of excision for singular chains,
we know of no analogy between the proof
we give here and standard proofs of excision in
singular homology, which use barycentric
subdivision as their fundamental tool.
The proof occupies the remainder of this section.

\begin{definition}
	Let $l\geqslant 0$ and let $a,b\in G\cup H$ be a 
	pair of vertices not both contained in $G$,
	and not both contained in $H$.  (Thus we must
	have $a\in G\setminus H$ and $b\in H\setminus G$,
	or \emph{vice versa}.)
	Define $A_{\ast,l}(a,b)$ to be the subcomplex
	of $\MC_{\ast,l}(G\cup H)$ spanned by those tuples
	$(x_0,\ldots,x_k)$ for which $x_0=a$, $x_k=b$,
	and $x_1,\ldots,x_{k-1}\in G\cap H$.
\end{definition}

\begin{lemma}
\label{A-acyclic:lemma}
	The complex $A_{\ast,l}(a,b)$ is acyclic.
\end{lemma}

\begin{proof}
For the purposes of the proof we assume that
$b\in H\setminus G$ and $a\in G\setminus H$,
the proof in the other case being similar.
Let us define a map
\[
	s\colon A_{\ast,l}(a,b)\longrightarrow A_{\ast+1,l}(a,b)
\]
by
\[
	s(x_0,\ldots,x_k) 
	=
	\begin{cases}
		(-1)^k(x_0,\ldots,x_{k-1},\pi(x_k),x_k)
		& \mathrm{if}\ x_{k-1}\neq\pi(x_k),
		\\
		0
		&
		\mathrm{if}\ x_{k-1}=\pi(x_k).
	\end{cases}
\]
We claim that $\partial\circ s + s\circ\partial=\mathrm{Id}$,
so that $s$ is a chain homotopy from $\mathrm{Id}$
to $0$, and in particular that $A_{\ast,l}(a,b)$ is acyclic.
Applied to a generator $(x_0,\ldots,x_k)$, this amounts 
to the claim that
\[
	\sum_{i=1}^k (-1)^i \partial_i s(x_0,\ldots,x_k)
	+ \sum_{i=1}^{k-1}(-1)^i s\partial_i(x_0,\ldots,x_k)
	=(x_0,\ldots,x_k).
\]
For $i=1,\ldots,k-2$ we have that
$\partial_i s(x_0,\ldots,x_k) +
s\partial_i(x_0,\ldots,x_k)=0$,
since $\partial_i$ does not affect the last two
entries and $s$ does not affect the first
$(k-1)$.  It therefore remains to show that
\[
	(-1)^{k-1}\partial_{k-1}s(x_0,\ldots,x_k)
	+
	(-1)^k \partial_k s(x_0,\ldots,x_k)
	+
	(-1)^{k-1}s\partial_{k-1}(x_0,\ldots,x_k)
\]
is equal to $(x_0,\ldots,x_k)$.
We verify this on a case-by-case basis.
\begin{itemize}
	\item
	If $x_{k-1}=\pi(x_k)$
	then  $x_{k-2}\neq \pi(x_k)$
	and 	$d(x_{k-2},x_{k-1})+d(x_{k-1},x_k) = d(x_{k-2},x_k)$.
	Consequently the first
	two terms in the above sum vanish, 
	while the third term is $(x_0,\ldots,x_k)$.
	\item
	If $x_{k-1}\neq \pi(x_k)$
	and
	$d(x_{k-2},x_{k-1})+d(x_{k-1},x_k) > d(x_{k-2},x_k)$,
	so that in addition 
	$d(x_{k-2},x_{k-1})+d(x_{k-1},\pi(x_k)) > d(x_{k-2},\pi(x_k))$,
	then the first and third terms in the sum vanish, 
	while the second is $(x_0,\ldots,x_k)$.
	\item
	If $x_{k-1}\neq \pi(x_k)$
	and 
	$d(x_{k-2},x_{k-1})+d(x_{k-1},x_k) = d(x_{k-2},x_k)$,
	so that in addition
	$d(x_{k-2},x_{k-1})+d(x_{k-1},\pi(x_k)) = d(x_{k-2},\pi(x_k))$
	and $x_{k-2}\neq \pi(x_k)$,
	then the sum above becomes
	\[
		-(x_0,\ldots,x_{k-2},\pi(x_k),x_k)
		+
		(x_0,\ldots,x_k)
		+
		(x_0,\ldots,x_{k-2},\pi(x_{k}),x_k).
	\]
\end{itemize}
In all cases the claim holds. This completes the proof.
\end{proof}

\begin{definition}
For what follows we require the following notion.
If $C_\ast$ is a chain complex and
$j\geqslant 0$, then the \emph{$j$-th suspension}
$\Sigma^{j}C_\ast$ of $C_\ast$ is the chain
complex in which $(\Sigma^j C_\ast)_i = C_{i-j}$.
\end{definition}

\begin{definition}
	Let $l\geqslant 0$.
	Given $b\in G\cup H \setminus G\cap H$,
	we define a complex $B_{\ast,l}(b)$
	and a subcomplex $\bar B_{\ast,l}(b)$
	as follows.
	If $b\in G\setminus H$ 
	then $B_{\ast,l}(b)$ is defined to be
	the subcomplex of
	$\MC_{\ast,l}(G\cup H)$ spanned by tuples of
	the form $(x_0,\ldots,x_k)$ with $x_k=b$
	and $x_0,\ldots x_{k-1}\in H$,
	and $\bar B_{\ast,l}(b)$ is defined to be the
	subcomplex of $B_{\ast,l}(b)$ spanned by
	tuples $(x_0,\ldots,x_k)$ in which
	$x_0,\ldots,x_{k-1}\in G\cap H$.
	If $b\in H\setminus G$ then the definitions
	are obtained in the same way,
	interchanging the role
	of $G$ and $H$.
\end{definition}

\begin{lemma}
\label{B-acyclic:lemma}
Let $l\geqslant 0$ and let $b\in G\cup H\setminus G\cap H$.
Then the complex $B_{\ast,l}(b)/\bar B_{\ast,l}(b)$
is acyclic.
\end{lemma}

\begin{proof}
Without loss we assume that $b\in H\setminus G$,
the proof in the other case being similar.
For $i=0,\ldots, l$ let $F_i$ denote the subcomplex
of $B_{\ast,l}(b)$ spanned by tuples $(x_0,\ldots,x_k)$
in which $x_{i},\ldots,x_{k-1}\in G\cap H$.
(In the case $i\geqslant k$ we impose no condition.)
Thus we obtain a filtration
\[
	\bar B_{\ast,l}(b) = F_0
	\subset F_1\subset \cdots
	\subset F_l=B_{\ast,l}(b)
\]
and it will suffice for us to show that for
each $i=1,\ldots,l$ the quotient $F_i/F_{i-1}$ is
acyclic.  
	
Let us describe the complex $F_i/F_{i-1}$.
Its generators are tuples $(x_0,\ldots,x_k)$
with $x_k=b$, with $x_i,\ldots,x_{k-1}\in G\cap H$,
and with $x_{i-1}\in G\setminus H$.
Here the first two conditions guarantee that
$(x_0,\ldots,x_k)$ is a generator of $F_i$,
while the third guarantees that it lies outside
$F_{i-1}$. 
The differential $\partial$ on $F_i/F_{i-1}$
is induced by the differential $\partial$ on
$F_i$, which is the alternating sum
$\sum_{i=1}^{k-1}(-1)^i \partial_i$ of the
operators $\partial_i$ which omit a generator's
$i$-th term if the length is preserved,
and otherwise send it to 0.
Reducing to $F_i/F_{i-1}$ we find that the
operators $\partial_1,\ldots,\partial_{i-1}$
become trivial, while $\partial_i,\ldots,\partial_{k-1}$
retain their previous description.

Using the description from the last paragraph,
we see that there is an isomorphism
	\[
		\bigoplus_{(x_0,\ldots,x_{i-1})}\Sigma^{i-1}
		A_{\ast,l-l'}(x_{i-1},b)
		\xrightarrow{\ \cong\ }
		F_i/F_{i-1}.
	\]
Here the direct sum is taken over all tuples
$(x_0,\ldots,x_{i-1})$ of elements of $H$ with 
$x_{i-1}\in H\setminus G$, and 
$l'=\ell(x_0,\ldots,x_{i-1})$. 
The isomorphism sends the generator $(x_{i-1},y_i,\ldots,y_k)$
of the summand $A_{k-i+1,l-l'}(x_{i-1},b)$ corresponding to 
$(x_0,\ldots,x_{i-1})$ to $(-1)^{(i-1)k}$ times 
the generator $(x_0,\ldots,x_{i-1},y_i,\ldots,y_k)$
of $F_i/F_{i-1}$.
This is a map of chain complexes 
since on $F_i/F_{i-1}$ in degree $k$ the maps 
$\partial_1,\ldots,\partial_{i-1}$ vanish
while the maps $\partial_i,\ldots,\partial_{k-1}$
are intertwined with the maps
$\partial_1,\ldots,\partial_{k-i}$ on
$A_{k-i+1,l-l'}(x_{i-1},b)$.
The map is an isomorphism since it restricts
to bijection between the generators of the domain and the range.
\end{proof}

\begin{proof}[Proof of Theorem~\ref{MV:Theorem}]
We wish to prove that the inclusion
\[
	\MC_{\ast,l}(G,H)
	\longrightarrow
	\MC_{\ast,l}(G\cup H)
\]
is a quasi-isomorphism.  
For $i=0,\ldots,l$ let $F_i$ denote the
subcomplex of $\MC_{\ast,l}(G\cup H)$ spanned by
the tuples $(x_0,\ldots, x_k)$ for which
$x_0,\ldots,x_{k-i}$ either all lie in $G$
or all lie in $H$.  (When $i>k$ we impose no
condition.)  Thus we have a filtration
\[
	F_0 \subset \cdots \subset F_l
\]
with
\[
	F_0
	=
	\MC_{\ast,l}(G,H)		
	\quad\mathrm{and}\quad	
	F_l = \MC_{\ast,l}(G\cup H).
\]
So it will suffice to prove that for $i=1,\ldots,l$
the quotient $F_i/F_{i-1}$ is contractible.
	
There is an isomorphism
\[
	\bigoplus_{(x_{k-i+1},\ldots,x_k)}
	\Sigma^{i-1}B_{\ast,l-l'}(x_{k-i+1})/\bar B_{\ast,l-l'}(x_{k-i+1})
	\xrightarrow{\ \cong\ }
	F_i/F_{i-1},
\]
where the direct sum is taken over all tuples
$(x_{k-i+1},\ldots,x_k)$ of elements of $G\cup H$
with $x_{k-i+1}\in G\cup H\setminus G\cap H$, and
$l' = \ell(x_{k_i+1},\ldots,x_k)$.
The isomorphism is given on the summand corresponding
to $(x_{k-i+1},\ldots,x_k)$ by sending a generator
$(x_0,\ldots,x_{k-i+1})$
to the generator $(x_0,\ldots,x_{k})$
of $F_i/F_{i-1}$.  
We omit the details of why this is an isomorphism;
the argument is similar to the one appearing in the
proof of Lemma~\ref{A-acyclic:lemma}.
Lemma~\ref{B-acyclic:lemma} shows that the domain of this
isomorphism is acyclic, and it follows that $F_i/F_{i-1}$
is acyclic.  This completes the proof.
\end{proof}

\section{Proof that joins are diagonal} 
\label{join:proof:section}

Let $G$ and $H$ be graphs satisfying $G,H\neq \emptyset$.
In this section we will prove Theorem~\ref{join:theorem},
which states that the join $G\join H$ is diagonal, or in other words
that $\MH_{k,l}(G\join H)=0$ for $k< l$.
We begin by stating the following, which is an immediate consequence
of the definition of $G\join H$.

\begin{lemma}
Let $a$ and $b$ be vertices of $G\join H$.
Then $d(a,b)$ can only take the values $0$, $1$ and $2$.
Moreover $d(a,b)=2$ only if $a$ and $b$ are both in $G$ or both in $H$.
\end{lemma}

For $i$ in the range $0\leqslant i\leqslant l-1$, let $F^i_\ast$ denote
the subcomplex of $\MC_{\ast,l}(G\join H)$ spanned by generators
$(x_0,\ldots,x_k)$ satisfying $d(x_j,x_{j+1})=2$ for some
$j\leqslant i$.  Thus
\[
	F^0_\ast\subset F^1_\ast\subset\cdots
	\subset F^{l-1}_\ast\subset \MC_{\ast,l}(G\join H).
\]
Observe that $F^{l-1}_\ast$ is simply the span of all generators
such that $k<l$.  Thus
\[
	F^{l-1}_k 
	= \begin{cases}
		\MC_{k,l}(G\join H) & \text{ if } k<l
		\\
		0 & \text{ if } k=l.
	\end{cases}
\]

\begin{lemma}\label{quotients:lemma}
For $i=0,\ldots,l-1$ the inclusion
$F^i_\ast/F^{i-1}_\ast \hookrightarrow \MC_{\ast,l}(G\join H)/F^{i-1}_\ast$
induces the zero map in homology.
\end{lemma}

\begin{proof}[Proof of Theorem~\ref{join:theorem}, 
assuming Lemma~\ref{quotients:lemma}]
Let us first prove by induction on $i=0,\ldots, l-1$
that the inclusion $F^i_\ast\hookrightarrow \MC_{\ast,l}(G\join H)$
induces the zero map in homology.
The initial case $i=0$ is an instance of Lemma~\ref{quotients:lemma}.
Assuming that the claim holds for $i$, let us prove it for $i+1$.
Since by hypothesis the inclusion 
$F^i_\ast\hookrightarrow \MC_{\ast,l}(G\join H)$
induces the zero map in homology, it follows that the quotient
$\MC_{\ast,l}(G\join H)\to \MC_{\ast,l}(G\join H)/F^i_\ast$
induces an injection in homology.  It will therefore suffice to
prove that the composite
\[
	F^{i+1}_\ast
	\to
	\MC_{\ast,l}(G\join H)
	\to 
	\MC_{\ast,l}(G\join H)/F^i_\ast
\]
induces the zero map in homology.  But this composite
can be rewritten as the composite
\[
	F^{i+1}_\ast
	\to
	F^{i+1}_\ast/F^i_\ast
	\to 
	\MC_{\ast,l}(G\join H)/F^i_\ast
\]
in which the second map induces the zero map in homology
by Lemma~\ref{quotients:lemma}.

Since $F^{l-1}_\ast \hookrightarrow \MC_{\ast,l}(G\join H)$
is an isomorphism in degrees $k<l$, and induces the zero
map in homology, it follows that $\MH_{k,l}(G\join H)=0$
for $k<l$.
\end{proof}

We now work towards a proof of Lemma~\ref{quotients:lemma}.
Given a vertex $x$ of $G\join H$,
denote by $A_\ast(x,l)$ the subcomplex
of $\MC_{\ast, l}(G\join H)$ generated by
the tuples of the form $(x,x_1,\ldots,x_k)$ 
with $d(x,x_1)=2$,
and denote by $B_\ast(x,l)$ the subcomplex
of $\MC_{\ast, l}(G\join H)$ generated
by the tuples of the form $(x,x_1,\ldots,x_k)$.

\begin{lemma}\label{commutative:lemma}
There is a commutative diagram
\[\xymatrix{
	\bigoplus\Sigma^iA_\ast(x_i,l-i)
	\ar@{^{(}->}[r]
	\ar[d]_\alpha^\cong
	&
	\bigoplus\Sigma^iB_\ast(x_i,l-i)
	\ar[d]^\beta
	\\
	F^i_\ast/F^{i-1}_\ast
	\ar@{^{(}->}[r]
	&
	\MC_{\ast,l}(G\join H)/F^{i-1}_\ast
}\]
in which the direct sums are indexed by tuples $(x_0,\ldots,x_i)$
of vertices of $G\join H$ satisfying $d(x_{j-1},x_j)=1$ for $j=1,\ldots,i$,
and in which the upper map is the direct sum of the inclusion maps.
\end{lemma}

\begin{proof}
We define $\alpha$ on the summand
corresponding to the tuple $(x_0,\ldots,x_i)$
to be the map
\[
	\bar\alpha\colon
	\Sigma^iA_\ast(x_i,l-i)
	\longrightarrow 
	F^i_\ast/F^{i-1}_\ast
\]
that sends a generator $(x_i,\ldots,x_k)$ 
to $(-1)^{ik}(x_0,\ldots,x_i,\ldots,x_k)$.
To see that $\bar\alpha$ is a chain map,
observe that in degree $k$ the differential
on $\Sigma^iA_\ast(x_i,l-i)$ is the sum 
$\sum_{j=1}^{k-i-1}(-1)^j\partial_j$,
while on $F^i_\ast/F^{i-1}_\ast$ it is the sum
$\sum_{j=1}^{k-1}(-1)^j\partial_j$.
However on $F^i_\ast/F^{i-1}_\ast$ the maps $\partial_1,\ldots,\partial_i$
all vanish, and in addition one can verify directly that
$\bar\alpha\circ\partial_j = (-1)^i\partial_{i+j}\circ\bar\alpha$ for
$j=1,\ldots,k-i-1$.
It follows that $\bar\alpha$ is indeed a chain map.
To see that $\alpha$ is an isomorphism, observe that the
generators of $F^i_\ast/F^{i-1}_\ast$ are precisely the
tuples $(x_0,\ldots,x_k)$ in which $d(x_0,x_1)=\cdots=d(x_{i-1},x_i)=1$
and $d(x_i,x_{i+1})=2$, 
so that $\alpha$ in fact restricts to a bijection
between the generators of its domain and range.
The chain map $\beta$ is obtained in an entirely analogous
way, and commutativity of the resulting square is then evident.
\end{proof}

\begin{lemma}\label{AB:lemma}
The inclusion $A_\ast(x,l)\hookrightarrow B_\ast(x,l)$
induces the trivial map in homology.
\end{lemma}

\begin{proof}[Proof of Lemma~\ref{quotients:lemma}, 
assuming Lemma~\ref{AB:lemma}]
Since the upper arrow of the commutative diagram of 
Lemma~\ref{commutative:lemma} induces the zero map
in homology, so does the lower arrow.
\end{proof}

We now work towards the proof of Lemma~\ref{AB:lemma}.
In order to do so, we assume without loss that $x\in G$, 
and we fix a vertex $y\in H$.
Then we define the \emph{height} of a generator 
$(x,x_1,\ldots,x_k)$ of $A_\ast(x,l)$
to be the largest integer $h$ such that
\[
	d(x,x_1)=2,\ 
	d(y,x_2)=2,\ 
	d(x,x_3)=2,\ 
	\ldots\ 
	d(-,x_h)=2
\]
where the final $-$ denotes $x$ if $h$ is odd and $y$ if $h$ is even.
Thus all generators have height at least $1$, and the height
of a generator is no more that its degree.

\begin{lemma}\label{height:lemma}
If $(x,x_1,\ldots,x_k)$ is a generator of $A_\ast(x,l)$
then $\partial_j(x,x_1,\ldots,x_k)$ is either $0$ or a
generator of height at most $j-1$.
\end{lemma}

\begin{proof}
Suppose not.  
Since $\partial_j(x,x_1,\ldots,x_k)$ is nonzero,
it follows that 
\[d(x_{j-1},x_j)=1,\  d(x_j,x_{j+1})=1\ 
\text{and}\  d(x_{j-1},x_{j+1})=2.\]
 In particular
$x_{j-1}$ and $x_{j+1}$ both lie in $G$ or 
both lie in $H$.
On the other hand, since $\partial_j(x,x_1,\ldots,x_k)$
has height at least $j$, then (assuming without loss
that $j$ is even) we have that 
$d(x,x_{j-1})=2$ and $d(y,x_{j+1})=2$, so that $x_{j-1}$
lies in $G$ and $x_{j+1}$ lies in $H$.  This is a contradiction.
\end{proof}

\begin{proof}[Proof of Lemma~\ref{AB:lemma}]
For $i\geqslant 1$, let $s_i\colon A_\ast(x,l)\to B_{\ast+1}(x,l)$
be the map defined on generators by the rule
\[
	s_i(x,x_1,\ldots,x_k)
	=
	\begin{cases}
		(\underbrace{x,y,x,y,\ldots,}_{i+1\text{ terms}}x_i,\ldots,x_k)
		&
		\text{if }i\leqslant h 
		\\
		0 
		&
		\text{if }i>h
	\end{cases}
\]
where $h$ denotes the height of $(x,x_1,\ldots,x_k)$.
In the first case, the assumption on the height 
guarantees that the term on the right has length exactly $l$.
We have the following compatibilities between the $s_i$
and the operators $\partial_j$.
\begin{itemize}
	\item
	$\partial_j\circ s_i = 0$ for $1\leqslant j<i$ 
	\item
	$\partial_{i+1}\circ s_i = \partial_{i+1}\circ s_{i+1}$
	for $1\leqslant i$
	\item
	$\partial_j\circ s_i = s_i\circ\partial_{j-1}$ for $i\geqslant 1$
	and $j\geqslant i+2$ 
	\item
	$\partial_1\circ s_1$ is the inclusion $A_\ast(x,l)\hookrightarrow
	B_\ast(x,l)$
	\item
	$s_i\circ\partial_j=0$ for $1\leqslant j\leqslant i$
\end{itemize}
The first four follow by direct computation and the last
follows from Lemma~\ref{height:lemma}.  

Now define $s\colon A_\ast(x,l)\to B_{\ast+1}(x,l)$ by
$s=\sum_{i\geqslant 1}(-1)^i s_i$.
We claim that $s$ is a chain homotopy from the inclusion map
to the zero map, or in other words that $s\circ\partial+\partial\circ s$
is the inclusion.  From the properties listed above we have
\[
	\partial\circ s
	= 
	\partial_1\circ s_1
	-\sum_{j\geqslant i+1,\,i\geqslant 1}(-1)^{i+j}s_i\circ\partial_j
\]
and 
\[
	s\circ\partial
	=
	\sum_{j\geqslant i+1, i\geqslant 1}(-1)^{i+j}s_i\circ\partial_j.
\]
It follows that $\partial\circ s + s\circ\partial = \partial_1\circ s_1$,
which is the inclusion map.
\end{proof}

\appendix
\section{Numerical examples}
\label{appendix}
In this appendix we give various further computer calculated 
examples of ranks of magnitude homology groups.  The reader can verify these computations using  
\verb!rational_graph_homology_arxiv.py!, written in sage and python, which was uploaded 
to the arXiv with this paper.

\subsection{Cyclic graphs}
\label{section:CyclicGraphs}
Cyclic graphs seem to have a clear pattern in their graph homology, and this pattern depends on whether the graph has an even or an odd number of vertices.  

First consider graphs with odd numbers of vertices as exemplified by the $7$-cyclic graph below.  (The 5-cyclic graph was given in Table~\ref{TableFiveCycle}.)
\begin{center}
\begin{tikzpicture}[baseline=0cm,scale=0.5]
\foreach \x in {0,51.43,...,308.58}
    \draw (\x+90:2cm) -- (\x+141.43:2cm);
%\draw (0:1cm) -- (180:1cm);
\foreach \x in {0,51.43,...,308.58}
    \draw [fill=red](\x+90:2cm) circle (0.1cm);
\end{tikzpicture}
\quad
\footnotesize
\begin{tabular}{rrrrrrrrrrrrr}
&0&1&2&3&4&5&6&7&8&9&10&11\\
\otoprule                                                                 
0 & 7\\
 1   & &     14                   \\                                          
 2     & &&         14             \\             
 3&     & &&         14             \\                                          
 4    &&    &      14  &&  14       \\                                          
 5         &&&&42& &  14  \\                                         
 6          &&&&&70&   & 14                          \\           
 7      &    &&&&&             98&   & 14                    \\           
 8          &&&&&28&&126&&  14              \\           
 9            &&&&&& 112 &&154&  &  14 \\            
10           &&&&&&& 252 && 182 & &   14 \\      
11           &&&&&&&&  448 && 210 &  &  14\\ 
\bottomrule
\smallskip
\end{tabular}
\end{center}
The behaviour for a cyclic graph with $n$ vertices where $n$ is odd,  seems to be as follows.  
The non-zero ranks are ordered in diagonal lines, so in the above the first diagonal starts at $k=0$, $l=0$ with $7,14,14,\dots$ and the second diagonal starting at $k=2$, $l=4$ with $14,42,70,\dots$: 
in general the $i$th diagonal starts at $k=2(i-1)$ and $l=(i-1)(n+1)/2$.   
Denote by $T^{n}_{i,j}$ the $j$th non-zero entry in the $i$th diagonal, then it would appear that these are given by the following recursion relation.
  \[T^{n}_{1,1}=n;\quad T^{n}_{1,2}=2n;\quad T^{n}_{i,j}=T^{n}_{i,j-1}+2 T^{n}_{i-1,j}.\]
%The ranks in the $i$th diagonal are all divisible by $n\cdot2^{i-1}$; 
%denote by $T_{i,j}$ the $j$th non-zero entry in the $i$th diagonal divided by  $n\cdot2^{i-1}$.  The resulting $T_{i,j}$ seem to be independent of $n$ and to satisfy a recursion relation.  From $n=5$ we the following table for $T_{i,j}$.
%\begin{center}\footnotesize
%\begin{tabular}{rrrrrrrrrrrrrrr}
%&&1&2&3&4&5&6&7&8&9&10&11\\
%\otoprule        
%1&&1&2&2&2&2&2&2&2&2&2&2\\
%2&&1&3&5&7&9&11&13&15&17\\
%3&&1&4&9&16&25&36\\
%4&&1&5&14\\
%\bottomrule
%\end{tabular}
%\smallskip
%\end{center}
%Another way of saying this is that
 %$T_{1,1}=1$ and $T_{1,j}=2$ for $j>1$, whilst $T_{i,j}=T_{i,j-1}+T_{i-1,j}$.
If that is the case then $T^{n}_{i,j}/(n 2^{i-1})$ is an integer independent of $n$.

On the other hand, the even case, as exemplified by the $8$-cycle graph below, appears more straightforward.
\begin{center}
\begin{tikzpicture}[baseline=0cm,scale=0.5]
\foreach \x in {0,45,...,315}
    \draw (\x+90:2cm) -- (\x+45:2cm);
%\draw (0:1cm) -- (180:1cm);
\foreach \x in {0,45,...,315}
    \draw [fill=red](\x+90:2cm) circle (0.1cm);
\end{tikzpicture}
\quad
\footnotesize
\begin{tabular}{rrrrrrrrrrrrr}
&0&1&2&3&4&5&6&7&8&9&10\\
\otoprule                                                                 
0 & 8\\
 1   & &     16                  \\                                          
 2     & &&         16             \\             
 3&     & &&         16             \\                                          
 4    &&    &      8  &&  16       \\                                          
 5         &&&&16& &  16  \\                                         
 6          &&&&&16&   & 16                          \\           
 7      &    &&&&&             16&   & 16                    \\           
 8          &&&&&  8 && 16 &&  16              \\           
 9            &&&&&& 16 &&  16  &  &  16 \\            
10           &&&&&&& 16 &&  16  & &   16 \\      
%&11           &&&&&&&&  16 && 16 &  &  16\\ 
\bottomrule
\smallskip
\end{tabular}
\end{center}

It appears that for an $n$-cycle graph with $n$ even, the first rank in each diagonal is $n$ and the subsequent ranks are $2n$.
The $i$th diagonal starts at $k=2(i-1)$ and $l=(i-1)n/2$.

\subsection{Projecting is necessary for Mayer-Vietoris}
\label{section:projectingnecessary}
This example shows the necessity of the projecting condition in the Mayer-Vietoris short exact sequence of a convex decomposition of a graph.
  Consider the graph $X$ pictured below.  
  This is the union of two $5$-cycle graphs along a common edge, i.e.~along a $2$-cycle.  
  This is a convex decomposition of the graph, however, neither $5$-cycle is projecting, as the `apex' of each $5$-cycle can't project.  
If this graph did have a Mayer-Vietoris short exact sequence then for each $k$ and $l$ we would have \[\rank\MH_{k,l}(X)=2\cdot\rank\MH_{k,l}(C_{5})-\rank\MH_{k,l}(C_{2}).\]  
The $2$-cycle is diagonal with $\rank\MH_{k,k}(C_{2})=2$ for all $k$.  
Comparing the table of ranks below with that for the $5$-cycle in Table~\ref{TableFiveCycle} we see that the first two diagonals are as you would expect if the above equation were satisfied, however the third diagonal is wrong, with the first differing entry being  $\rank\MH_{2,4}(X)=2$.
%\begin{table}[h]
\begin{center}
\begin{tikzpicture}[baseline=0cm,scale=0.5]
\foreach \x in {0,45,...,315}
    \draw (\x:2cm) -- (\x+45:2cm);
\draw (0:2cm) -- (180:2cm);
\foreach \x in {0,45,...,315}
    \draw [fill=red](\x:2cm) circle (0.1cm);
\end{tikzpicture}
\quad
\footnotesize
\begin{tabular}{rrrrrrrrrrrr}
&0&1&2&3&4&5&6&7&8&9&10\\
\otoprule
 0&   8                                                             \\
 1  &&       18                                                     \\  
 2      &&&         18                                          \\       
 3             &&&  20   & 18                                  \\         
 4              &&&  2  &  60  &  18                           \\          
 5             &&&&        12  & 100 &   18                   \\            
 6                 &&&&&          76  & 140  &  18             \\            
 7             &&&&&               8  & 236  & 180  &  18     \\              
 8                &&&&&            2  &  56 &  492 &  220  &  18 \\            
 9                       &&&&&&          16  & 280  & 844  & 260 &   18\\       
10                          &&&&&&&             92  & 904&  1292 &  300  &  18 \\
\bottomrule
\smallskip
\end{tabular}
\end{center}
%\caption{The ranks of the magnitude homology groups of the pictured graph}
%\end{table}
%

\subsection{Some symmetric cubic graphs} 
Here we include some classical graphs for further examples.  
These graphs all have large symmetry groups, and 
these act on the magnitude homology groups.
Intriguingly the order of the automorphism group is showing up in the ranks of the homology group.
This might be simply indicating that the automorphism group is acting freely transitively on the generators of those magnitude homology groups.

Some of these graphs have patterns in the ranks of the magnitude homology groups reminiscent of those for the cyclic graphs.  We leave the reader to discover them.

The M\"obius Kantor graph and the Pappus graph illustrate that the rank can sometimes decrease as you move down a diagonal.
\subsubsection{Petersen graph} The automorphism group has order $120$.
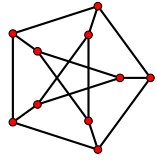
\begin{center}
\begin{tikzpicture}[baseline=0cm]
  \tikzset{VertexStyle/.style= {shape=circle,  color=black, fill=red, inner sep=0pt, 
               minimum size = 0.1cm, draw}}
   \SetVertexNoLabel
   \tikzset{EdgeStyle/.style= {thick}} 
   \grPetersen[form=1,RA=1,RB=0.6]%
\end{tikzpicture}\quad
\footnotesize
\begin{tabular}{rrrrrrrrrr}
&0&1&2&3&4&5&6&7&8\\
\otoprule                                                                 
0 & 10\\
 1   & &     30                   \\                                          
 2     & &&         30             \\                                          
 3    &&&           120  &  30       \\                                          
 4         &&&&            480  &  30  \\                                         
 5          &&&&&                 840   & 30                          \\           
 6          &&&&&                 1440  &  1200   & 30                    \\           
 7           &&&&&&                      7200  &  1560  &  30              \\           
 8           &&&&&&&                           17280  & 1920 &   30      \\              
\bottomrule
\smallskip
\end{tabular}
\end{center}

\subsubsection{Heawood graph} The automorphism group has order $336$.
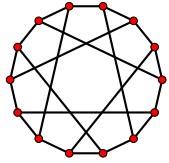
\begin{center}
\begin{tikzpicture}[baseline=0cm]
  \tikzset{VertexStyle/.style= {shape=circle,  color=black, fill=red, inner sep=0pt, 
               minimum size = 0.1cm, draw}}
   \SetVertexNoLabel
   \tikzset{EdgeStyle/.style= {thick}} 
   \grHeawood[RA=1]
\end{tikzpicture}\quad
\footnotesize
\begin{tabular}{rrrrrrrrrr}
&0&1&2&3&4&5&6&7&8\\
\otoprule                                                                 
0 & 14\\
 1   & &     42                   \\                                          
 2     & &&         42             \\                                          
 3    &&&           112  &  42       \\                                          
 4         &&&&            336  &  42  \\                                         
 5          &&&&&                 336   & 42                          \\           
 6          &&&&&                 896  &  336   & 42                    \\           
 7           &&&&&&                      2688  &  336  &  42              \\           
 8           &&&&&&&                           2688  & 336 &   42      \\              
\bottomrule
\smallskip
\end{tabular}
\end{center}

\subsubsection{Tutte Coxeter graph} The automorphism group has order $1440$.
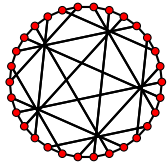
\begin{center}
\begin{tikzpicture}[baseline=0cm]
  \tikzset{VertexStyle/.style= {shape=circle,  color=black, fill=red, inner sep=0pt, 
               minimum size = 0.1cm, draw}}
   \SetVertexNoLabel
   \tikzset{EdgeStyle/.style= {thick}} 
   \grLCF[RA=1]{-13,-9,7,-7,9,13}{5}
\end{tikzpicture}\quad
\footnotesize
\begin{tabular}{rrrrrrrrrrr}
&0&1&2&3&4&5&6&7&8\\
\otoprule                                                                 
0 & 30\\
 1   & &     90                  \\                                          
 2     & &&         90             \\             
 3&     & &&         90             \\                                          
 4    &&    &      480  &&  90       \\                                          
 5         &&&& 1440 & &  90  \\                                         
6          &&&&& 1440 &   & 90                          \\           
 7      &    &&&&&              1440 &   & \,\,\,90                  \\           
 8          &&&&&  7680 &&  1440  &&  \,\,\,90              \\           
\bottomrule
\smallskip
\end{tabular}
\end{center}

\enlargethispage*{4em}
\subsubsection{Moebius Kantor graph} The automorphism group has order $96$.
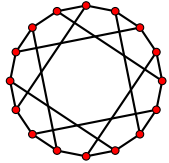
\begin{center}
\begin{tikzpicture}[baseline=0cm]
   \tikzset{VertexStyle/.style= {shape        = circle,
   color=black,fill=red,inner sep=0pt,
%                                      shading      = ball,
%                                      ball color   = blue!60,
                                       minimum size = 0.1cm,
                                  draw}}
   \SetVertexNoLabel
   \tikzset{EdgeStyle/.style= {thick}}
   \grMobiusKantor[RA=1]
\end{tikzpicture}\quad
\footnotesize
\begin{tabular}{rrrrrrrrrr}
&0&1&2&3&4&5&6&7&8\\
\otoprule                                                                 
0 & 16\\
 1   & &     48                   \\                                          
 2     & &&         48             \\                                          
 3    &&&           112  &  48       \\                                          
 4         &&&&            304  &  48  \\                                         
 5          &&&&          48&       288   & 48                          \\           
 6          &&&&&                 832  &  288   & 48                    \\           
 7           &&&&&&                      1952  &  288  &  48              \\           
 8           &&&&&&               656&         1776  & 288 &   48      \\              
\bottomrule
\smallskip
\end{tabular}
\end{center}

\subsubsection{Pappus graph} The automorphism group has order $216$.
\begin{center}
\begin{tikzpicture}[baseline=0cm]
\tikzset{VertexStyle/.style= {shape        = circle,
   color=black,fill=red,inner sep=0pt, minimum size = 0.1cm, draw}}
   \SetVertexNoLabel
   \tikzset{EdgeStyle/.style= {thick}}  
    \grPappus[RA=1]
\end{tikzpicture}\quad
\footnotesize
\begin{tabular}{rrrrrrrrrr}
&0&1&2&3&4&5&6&7&8\\
\otoprule                                                                 
0 & 18\\
 1   & &     54                   \\                                          
 2     & &&         54             \\                                          
 3    &&&           108  &  54       \\                                          
 4         &&&&            252  &  54  \\                                         
 5          &&&&          108&       216   & 54                          \\           
 6          &&&&&                 756  &  216   & 54                    \\           
 7           &&&&&&                      1188  &  216  &  54              \\           
 8           &&&&&&               1224&         972  & 216 &   54      \\              
\bottomrule
\smallskip
\end{tabular}
\end{center}

\subsection{The icosahedral graph} 
\label{section:icosahedron}
This is the only graph which our calculations show to be diagonal, but for which we know of no proof that it is diagonal.
\begin{center}
\begin{tikzpicture}[baseline=0cm]
\tikzset{VertexStyle/.style= {shape        = circle,
   color=black,fill=red,inner sep=0pt, minimum size = 0.1cm, draw}}
   \SetVertexNoLabel
   \tikzset{EdgeStyle/.style= {thick}}   
   \grIcosahedral[form=1,RA=1,RB=0.5]
\end{tikzpicture}\quad
\footnotesize
\begin{tabular}{rrrrrrrrr}
&0&1&2&3&4&5&6&7\\
\otoprule                                                                 
0 & 12\\
 1   & &     60                   \\                                          
 2     & &&         240             \\                                          
 3    &&&             &  912       \\                                          
 4         &&&&              &  3420  \\                                         
 5          &&&&          &          & 12780                          \\           
 6          &&&&&                   &     & 47712                    \\           
 7           &&&&&&                        &    &  178080              \\           
            
\bottomrule
\smallskip
\end{tabular}
\end{center}

\bibliographystyle{plain}
\bibliography{MagnitudeHomology}
\end{document}